 \documentclass[12pt]{amsart}

\usepackage{amsmath}
\usepackage{amssymb} 
\usepackage{array}
\usepackage{enumitem}
\usepackage{hyperref}
\usepackage{verbatim} 
\usepackage{color}
\definecolor{battleshipgrey}{rgb}{0.52, 0.52, 0.51} 
\hypersetup{
       colorlinks=true,       
    linkcolor=battleshipgrey,     
    citecolor=battleshipgrey,        
    filecolor=magenta,      
    urlcolor=battleshipgrey          
}

\theoremstyle{plain}
\newtheorem{theorem}{Theorem}[section]
\newtheorem{lemma}[theorem]{Lemma}

\newtheorem{definition}[theorem]{Definition}
\theoremstyle{remark}
\newtheorem{remark}{Remark}[section]
\newtheorem{example}{Example}[section]
\newtheorem*{notation}{Notation}
\newtheorem*{acknowledgment}{Acknowledgment}

\numberwithin{equation}{section}

\newcommand{\bA}{\mathbb{A}}

\newcommand{\bB}{\mathbb{B}}

\newcommand{\K}{\mathbb{K}}

\newcommand{\R}{\mathbb{R}}

\newcommand{\N}{\mathbb{N}}

\newcommand{\cP}{{\mathcal P}}

\newcommand{\bftwo}{\mathbf{2}}

\newcommand{\sss}{\mathbf{s}}

\newcommand{\ttt}{\mathbf{t}}

\newcommand{\vvv}{\mathbf{v}}

\newcommand{\Mod}{\mathbf{Mod}}
\newcommand{\Cu}{\mathbf{CubeSet}}
\newcommand{\Space}{\mathbf{Space}}

\newcommand{\sfn}{\mathsf{n}} 

\newcommand{\sfone}{\mathsf{1}}
\newcommand{\sftwo}{\mathsf{2}}

\newcommand{\ev}{\mathrm{ev}}

\newcommand{\id}{\mathrm{id}}

\newcommand{\Talg}{\mathbf{talg}} 
\newcommand{\Scal}{\mathbf{scal}} 
\newcommand{\SET}{\mathbf{Sets}} 
\newcommand{\LL}{\mathbf{Llin}} 
\newcommand{\Fn}{\mathbf{Fn}} 
\newcommand{\bfc}{\mathbf{c}}

\newcommand{\eps}{\varepsilon}

\newcommand{\ka}{\kappa}
\newcommand{\Ups}{\Upsilon}
\newcommand{\sett}[1]{{\{ #1 \}}} 

\newcommand{\msk}{\medskip}
\newcommand{\ssk}{\smallskip}
\newcommand{\nin}{\noindent}

\newcommand{\ul}{\underline}
\newcommand{\ull}[1]{\underline{\underline{#1}}} 

\newcommand\x{\times}

\begin{document}

\title[A Functorial Approach to Differential Calculus]{A Functorial Approach to Differential Calculus}

\author{Wolfgang Bertram, J\'er\'emy Haut}

\address{Institut \'{E}lie Cartan de Lorraine \\
Universit\'{e} de Lorraine at Nancy, CNRS, INRIA \\
Boulevard des Aiguillettes, B.P. 239 \\
F-54506 Vand\oe{}uvre-l\`{e}s-Nancy, France}

\email{\url{wolfgang.bertram, jeremy.haut@univ-lorraine.fr}}

\subjclass[2010]{18A25 ,  
18B40 , 
18D05 , 
58A05  
}

\keywords{ differential calculus, functor category, anchor, 
tangent algebra
}

\begin{abstract}  
We show that differential calculus (in its usual form, or in the
general form of {\em topological differential calculus}) can be fully imdedded
into a functor category (functors from a small category of {\em anchored 
tangent algebras} to anchored sets). 
To prepare this approach,
we define a new, symmetric, presentation of 
differential calculus, 
whose main feature is the central r\^ole played by the  {\em anchor
map}, which we study in detail.
The results rely heavily on the {\em monoidal structure} of the
category of anchored tangent algebras, given by the tensor
product of algebras. 
\end{abstract}

\maketitle

\setcounter{tocdepth}{1}
\tableofcontents

\section*{Introduction}

Differential Calculus is a central ingredient of modern mathematics. While the ``working mathematician'' takes this tool for granted,  thinking about its conceptual foundations remains a potentially important topic. In the present work, 
we continue the line of research started with \cite{BGN04, Be08, BeS14, Be17}, and combine it with what Grothendieck once called  the
``simple idea of a good functor from rings to sets'' 
(according to W.\ Lawvere, cf.\  \href{https://ncatlab.org/nlab/show/functorial+geometry#Lawvere03}{\tt n-lab})\footnote{Here the quote 
from the $n$-lab: 
``The 1973 Buffalo Colloquium talk by Alexander Grothendieck had as its main theme that the 1960 definition of scheme ...
should be abandoned AS the FUNDAMENTAL one and replaced by the {\em simple idea of a good functor from rings to sets}. The needed restrictions could be more intuitively and more geometrically stated directly in terms of the topos of such functors, and of course the ingredients from the ``baggage''  could be extracted when needed as auxiliary explanations of already existing objects, rather than being carried always as core elements of the very definition.''}.
The ``simple idea'' mentioned by Grothendieck is currently used 
in algebraic geometry, and in Lie Theory, where one often 
considers a real ``space'' -- for instance, a Lie group $\ul G$ -- 
as set of ``real points'' $G_\R$ of  a {\em complex} Lie group $G_{\mathbb C}$. 
This 
is a kind of non-linear analog of the complexification
$V_{\mathbb C} =  V \otimes_\R {\mathbb C}$ of a real vector space (or of a
real Lie algebra).
Grothendieck's insight was that this idea of ``complexification'' should not be limited to {\em field} extensions, but enlarged to more general
{\em ring} extensions, in order to incorporate operations belonging to {\em infinitesimal calculus}: a $\K$-Lie group $G$, or a general $\K$-smooth manifold
$M$, should admit ``scalar extensions''
$M_\bA$ akin to a hypothetic tensor product
 $M \otimes_\K \bA$, for certain $\K$-algebras $\bA$. 
The simplest example of such an extension is the one
by {\em dual numbers},
\begin{equation}
\label{eqn:dual}
\K[\eps] := \K[X]/(X^2) = \K \oplus \eps \K \quad (\eps^2 =0) ,
\end{equation}
where the nilpotent element $\eps$ is the class $[X]$ modulo $(X^2)$. 
Grothendieck, following ideas of Weil \cite{We53}, realized that the tangent bundle $TM$ of a ``space'' $M$, which is ``defined over $\K$'',
could be understood as something like $M \otimes_\K \K[\eps]$. 
 This idea has been used 
 by Demazure and Gabriel in their theory of algebraic groups \cite{DG}, 
 in differential calculus over general base field and rings \cite{Be08}, 
 and in the approach to natural operations in differential geometry via the so-called {\em Weil functors} (\cite{KMS93}, cf.\ also \cite{BeS14}).
The most elaborate and systematic development of these ideas leads to what is called nowadays
\href{https://ncatlab.org/nlab/show/synthetic+differential+geometry}{\em synthetic differential geometry} (SDG, see \cite{MR}). The approach to be
presented here pursues the same goals as SDG, but by different means:
we keep closer to the idea of generalizing the algebraic tensor product.
In a very direct sense, our problem is to generalize the algebraic scalar extension
$V_\bA := V \otimes_\K \bA$ of a $\K$-module $V$, 
to more general
spaces $M$, like, e.g., manifolds -- where we face the problem that such
an operation won't be possible for {\em all} $\K$-algebras $\bA$, so we have
to single out a good class (good category)
 of algebras for which such an extension is
possible. Such a class, called {\em the category of (anchored) tangent
algebras}, will be defined in this paper.
It arises naturally, when questioning  the very shape of differential calculus, instead of taking it for granted.
Let us briefly explain the main ideas.

\subsection{Topological differential calculus}
In differential calculus we consider maps $f$ whose domain $U$ and codomain $U'$ are  {\em locally linear sets} -- 
by this we mean $U \subset V$ and $U' \subset V'$ are non-empty subsets of linear (or affine, if one prefers)  spaces $V$ and $V'$. 
In this situation, we may define the {\em slope} or {\em difference quotient map}:
 when $t,s \in \K$ are such that $t-s$ is invertible, we look at the 
 difference quotient
\begin{equation}\label{eqn:slope} 
f^{[1]}(v_0,v_1; t,s) := 
f^{[1]}_{(t,s)}(v_0,v_1) :=
  \frac{ f(v_0 + tv_1) - f(v_0 + sv_1)}{t-s} .
\end{equation}
To speak of {\em topological} calculus, we shall assume that
 $V,V'$ are topological vector spaces or modules over
topological fields or rings $\K$, and $U,U'$ are open. For the moment, let's
consider the ``classical case''  $\K=\R$ and $V = \R^n$, $V' = \R^m$.
Then the following holds (cf.\ \cite{BGN04, Be08}):
{\em The map $f$ is of class $C^1$ if, and only if, the 
difference quotient 
map $f^{[1]}$ extends continuously to a map defined on the set}
\begin{equation}
\label{eqn:U-one}
U^{[1]} := \Bigl\{ (v_0,v_1;t,s) \in V^2 \times \K^2 \Big|\, \,   \, \begin{matrix} 
 v_0 + tv_1 \in U \\  v_0 + sv_1 \in U \end{matrix} \, \,  \Bigr\}.
\end{equation}
If this is the case, we denote still by $f^{[1]}:U^{[1]} \to U'$ the extended map.
Then
$f^{[1]}(v_0,v_1;0,0) = df(v_0)v_1$  gives the differential $d f$ of $f$.
Now, these conditions make perfectly sense for any ``good'' topological ring
$\K$ and 
for maps defined on open locally linear sets, and thus can serve as {\em definition}
of differentiability over $\K$  --
the ``topological differential calculus'' thus defined has excellent functorial
properties allowing to give a ``purely algebraic'' presentation of certain features
of usual calculus (see \cite{BGN04, Be11}).  To understand the structure of
formulae like (\ref{eqn:slope}) and (\ref{eqn:U-one}),
the following {\em way of talking} turns out to be useful: 
\begin{itemize}
\item call $\vvv = (v_0,v_1)$ ``space variables'', with
$v_0$ the ``foot point'' and $v_1$ the ``direction'' (in which we differentiate), 
\item
call $(t,s)$  ``time variables'', and
$t$  ``target time'', and $s$  ``source time'', 
\item
call $(t,s)$ ``regular'', or ``finite'', if $t-s$ is invertible in $\K$, and  ``singular''
or ``infinitesimal'' else, with $t-s=0$ being the ``most singular value'',
\item call 
$v_0 + s v_1$ the ``source'',  and $v_0 + tv_1$ the ``target evaluation point'',
\item
for fixed $(t,s)$,
call
$\alpha\bigl( (v_0,v_1) \bigr) := v_0 + s v_1$ the  ``source map'', and define the ``target map''
$\beta\bigl( (v_0,v_1) \bigr) := v_0 + t v_1$ .
\end{itemize}
The slogan summarizing topological calculus  is:  {\em the slope extends 
continuously (jointly in space and time variables) from finite  to singular
times}.
The notable difference with \cite{BGN04, Be11} is  that here we shall use
a {\em pair} of  time
parameters $(t,s)$, instead of a single parameter $t$ as in loc.\ cit. 
Although the expression (\ref{eqn:slope}) 
is of course symmetric under switch of target and source time,
it will be important to distinguish ``target'' and ``source''. 
The setting of \cite{BGN04, Be11} is  gotten by restricting to $s=0$ (we call this ``target calculus''); symmetrically, the theory could also be
formulated when letting $t=0$ (``source calculus'').
But now we can take advantage to define a third calculus, the ``symmetric
calculus'', which corresponds to the case $t=-s$: then
$v_0 = \frac{v_0 + s v_1 + v_0 + t v_1}{2}$, so 
the footpoint is the midpoint of target and source evaluation point --
see Subsection \ref{ssec:symmetric}.\footnote{ A
price has to be paid: one will have to require that $2$ be invertible in $\K$.
Analysts won't bother, some algebraists might... }

 \subsection{The underlying algebraic structure: anchor}\label{ssec:anchor}
In the second section we shall  carve out the algebraic structures underlying topological differential calculus.
As in general groupoid theory, the pair $(\alpha,\beta)$ given by source and
target will be called
{\em anchor map}\footnote{ This map is indeed the anchor map of a groupoid structure, see Subsection  \ref{ssec:groupoid}. }.
We use the same term when considering the pair of  time variables
$(t,s)$ as a ``frozen parameter'' (temporarily considered to be fixed); then we write $(t,s)$ as lower index -- for instance,
\begin{equation}\label{eqn:U_ts}
U^{[1]}_{(t,s)} := \{ (v_0,v_1) \mid \, (v_0,v_1;t,s) \in U^{[1]} \} .
\end{equation}
For fixed $(t,s)$, we call again 
{\em anchor}
 the (linear) map sending the space variables
 $\vvv = (v_0,v_1)$ to the pair of evaluation points:
\begin{equation} \label{eqn:anchor}
\Upsilon_{(t,s)} : U^{[1]}_{(t,s)} \to U \times U, \quad 
\begin{pmatrix}v_0 \\v_1 \end{pmatrix}
 \mapsto
 \begin{pmatrix}x_0 \\x_1 \end{pmatrix}
=
  \begin{pmatrix} 1 & s \\ 1 & t \end{pmatrix}  \begin{pmatrix}v_0 \\v_1 \end{pmatrix} =
 \begin{pmatrix}v_0 +s v_1 \\v_0 + t v_1  \end{pmatrix} 
 =  \begin{pmatrix} \alpha(\vvv)  \\ \beta(\vvv)  \end{pmatrix} .
\end{equation}
Of course, a choice is made here: 
the ``first'' component of $U \times U$ shall be associated with
``source'', and the ``second'' with ``target''.  One of our concerns in the sequel 
will be to formalize the levels on which such choices are operated. 
Anyhow, by direct computation, the anchor
is seen to be invertible if, and only if, $t-s$ is invertible, and then its inverse is given by
\begin{equation} 
\label{eqn:anchor-inverse} 
\Upsilon^{-1}_{(t,s)} : U\times U \to 
U^{[1]}_{(t,s)} , \quad 
\begin{pmatrix}x_0 \\x_1 \end{pmatrix}
\mapsto
 \frac{1}{t-s} 
 \begin{pmatrix} t & -s \\  - 1 & 1 \end{pmatrix}  \begin{pmatrix}x_0 \\ x_1 \end{pmatrix} =
 \begin{pmatrix}  \frac{ t x_0 - s x_1}{t-s}  \\ \frac{x_1 - x_0}{t-s}\end{pmatrix}.
\end{equation} 
The first component is an affine combination 
$v_0=\frac{s}{s-t} x_1 + \frac{t}{t-s} x_0$, and the second a ``difference quotient''.
From this, comparing with (\ref{eqn:slope}), we see that $f^{[1]}_{(t,s)}$ is precisely the second component of the map
$f^\sett{1}_{(t,s)} := \Upsilon^{-1}_{(t,s)} \circ (f\times f) \circ \Upsilon_{(t,s)}$, 
given by
\begin{equation}
\label{eqn:f_ts} 
f^\sett{1}_{(t,s)}  \begin{pmatrix}v_0 \\v_1 \end{pmatrix} =
\begin{pmatrix} 
\frac{ t f(v_0 + sv_1) - s f(v_0 + t v_1)}{t-s}  \\     
\frac{ f(v_0 + tv_1) - f(v_0 + sv_1)}{t-s}   \end{pmatrix} .
\end{equation}
The big advantage is that $f^\sett{1}_{(t,s)}$ {\em depends
functorially on $f$}: the ``chain rule'' simply reads
$(g \circ f)^\sett{1}_{(t,s)} = g^\sett{1}_{(t,s)} \circ f^\sett{1}_{(t,s)}$. 
Now we can  reformulate the property of being $C^1_\K$ (Lemma \ref{la:C1}):
{\em 
The map $f:U \to U'$ is of class $C^1_\K$ if, and only if, 
for all $(t,s) \in \K^2$
there exists a continuous map
$f^\sett{1}_{(t,s)} : U^\sett{1}_{(t,s)}  \to (U')^\sett{1}_{(t,s)}$, jointly continuous
also in the parameter $(t,s) \in \K^2$, such that }
\begin{equation} \label{eqn:anchor1}
\Upsilon_{(t,s)} \circ  f^\sett{1}_{(t,s)}  = (f \times f) \circ \Upsilon_{(t,s)} :
\qquad  \begin{matrix} U_{(t,s)} &
 {\buildrel{f^\sett{1}_{(t,s)}} \over  \longrightarrow} & U_{(t,s)}'  \\
\Ups \downarrow \phantom{\Ups} & & \phantom{\Ups}\downarrow \Ups \\
 U\times U & {\buildrel{f \times f} \over  \longrightarrow}  & U \times U
 \end{matrix}
\end{equation}
In a nutshell, this diagram contains the essential ingredients needed for
our approach: our aim is to translate diagram (\ref{eqn:anchor1})
into a ``categorical'' formulation,
so that it will make sense in an abstract setting, not requiring topology
any more. In a first step, we generalize this diagram at higher order
$n\in \N$  (Theorem \ref{th:Cn}): 
indeed, differentiability at order $n$ is characterized by a diagram of
the same kind, replacing $f^\sett{1}_{(t,s)}$, etc., 
by higher order maps $f^\sett{\sfn}_{(\ttt,\sss)}$, etc., where 
 $(\ttt,\sss)=(t_1,\ldots,t_n;s_1,\ldots,s_n) \in \K^{2n}$.
Technically, we work with $2^n$-fold direct products, which have to
be indexed by elements $A$ of the {\em $n$-hypercube}
$\cP(\sfn)$ (power set of $\sfn = \{ 1,\ldots,n\}$).

\subsection{The simple idea of a good functor from rings to sets}
In order to formalize the idea that the extended domains and maps
$(U^\sett{\sfn}_{(\ttt,\sss)}, f^\sett{\sfn}_{(\ttt,\sss)})$ are scalar
extensions $(U \otimes_\K \bA, f\otimes_\K \bA)$, we look at the
case $U =\K$. 
From functoriality, it follows that the spaces
$\K^\sett{\sfn}_{(\ttt,\sss)}$ are in fact {\em $\K$-algebras}, which can
easily be identified, 
\begin{enumerate}
\item
in terms of polynomial rings: they are polynomial algebras
$\K[X_1,\ldots,X_n]$, quotiented by the relations
$(X_i - t_i)(X_i - s_i) = 0$, for $i=1,\ldots,n$,
\item
in terms of tensor products: they are $n$-fold tensor products of ``first order
algebras''
$\K_{(t_1,s_1)} \otimes \ldots \otimes \K_{(t_n,s_n)}$.
\end{enumerate}
The second item shows that the collection of these algebras $\K_{(\ttt,\sss)}^\sfn$
forms a {\em small monoidal category} with respect to the tensor product,
where we define morphisms to be given by left or right multiplications coming
from the monoid structure. This is the {\em category $\Talg_\K$ of $\K$-tangent
algebras}.
Every such algebra admits an {\em anchor morphism}
$\Ups^\sfn_{(\ttt,\sss)} : \K_{(\ttt,\sss)}^\sfn \to \K^{\cP(\sfn)}$ to the {\em cube algebra}
which is a direct product of copies of $\K$, indexed by the $n$-hypercube
$\cP(\sfn)$. 
We compute 
 an explicit formula describing $\Ups^\sfn_{(\ttt,\sss)}$
(Theorem \ref{th:Anchor}). This anchor morphism is an isomorphism if, and
only if,  $(\ttt,\sss)$ is {\em regular}, and we give an explicit formula for the
inverse morphism (Theorem \ref{th:Anchor-inverse}).

\ssk
Now, the ``simple idea of a good functor from rings to sets'' is to view ``$\K$-smooth spaces'' as functors $\ul M$ from the 
category $\Talg_{\K}$ to the category of sets,
satisfying certain conditions specified in Subsection \ref{ssec:K-space},
and ``$\K$-smooth maps'' as certain {\em natural transformations} 
between functors $\ul M$ and $\ul M'$,
behaving in all respects like a family of ``algebraic scalar extensions''
$f \otimes_\K \id_{\K_{(\ttt,\sss)}^\sfn}$.
Indeed, in the framework of topological differential calculus, 
for a smooth map $f: M \to M'$, the family 
$f^\sfn_{(\ttt,\sss)}$ satisfies these conditions, and thus 
``topological calculus'' imbeds into ``categorical calculus''. 

\ssk
In order to fully justify such a functorial
 approach to differential calculus, one usually requires in
 SDG that the model be {\em well-adapted}, that is, that we obtain a
 {\em full and faithful} imbedding of a ``usual'' category of differential
calculus into the ``functorial'' one. We show that, for our setting, this is indeed
the case (Theorem \ref{th:imbedding2}).
The proof is much easier than the one of analogs in SDG, because, in essence,
the whole setting is designed for such a theorem to hold: 
it is merely the translation of Theorem \ref{th:Cn} into a more abstract language.

\subsection{Further topics}
The aim of this work is to lay the basic framework for a purely categorical
approach to calculus over general (commutative) base rings.
In Section \ref{sec:further} we briefly indicate further questions and topics
to be studied in this context:
to study {\em natural transformations} in the sense of \cite{KMS93}, we have
to introduce further morphisms in our categories, and in particular those
arising via the natural {\em (higher order) groupoid structure} that exists
on the algebras $\K_{(\ttt,\sss)}^\sfn$.
Very likely, a good understanding requires to understand also the
{\em full} iteration procedure, and not only the restricted one used here,
so to include, for instance, also the {\em simplicial calculus} from \cite{Be13}.
Finally, we conjecture that, replacing the usual braiding of tensor products
by the braiding defining the {\em graded tensor product}, the present approach
will also turn out to be useful in a categorical approach to {\em super-calculus}.

\begin{acknowledgment}
Part of these results should have been presented at the
\href{https://www.univ-saida.dz/cimpagal2020/}{CIMPA spring school
``Lie groupoids and algebroids''},
which had to be cancelled due to the Covid-19 crisis.
We thank the organisers for their work, and we hope that
the school will take place soon after the end of this crisis.
We also thank Alain Genestier for helpful discussions and the
referee for his careful reading and useful comments on the manuscript.
\end{acknowledgment}

\begin{notation}
We write $\N = \{ 1,2,\ldots \}$ and $\N_0 = \N\cup \{ 0 \}$, and let
$\sfn = \sett{1,2,\ldots,n}$.
Categories are denoted in boldface characters: small letters 
for small categories, such as $\Talg_\K$, and capital letters for large
categories, such as $\SET$ (category of sets).
The letter $\Fn$ stands for ``functor category'', so
$\Fn(\bfc,\SET) = \SET^\bfc$ is the category of (covariant) functors from a
(small) category $\bfc$ to $\SET$. 
Throughout, $\K$ is a commutative  base ring with unit $1$. 
\end{notation}

\section{Topological differential calculus}\label{sec:topcal}

In differential calculus, one usually is
mostly interested in the {\em morphisms}, that is, in {\em maps of class $C^n$}.
However, let us first say some words about the {\em objects}:

\subsection{Locally linear sets, and the anchor} 
A {\em locally linear set} is a pair $(U,V)$, where  $V$ is  a $\K$-module, and $U \subset V$ a non-empty subset.
We define the set $U^{[1]}$ by (\ref{eqn:U-one}), and the  {\em 
(full) anchor} by
\begin{equation}
\Ups : U^{[1]} \to (U \times \K)^2, \quad
(v_0,v_1;t,s) \mapsto 
(v_0 + s v_1,s ; v_0 + t v_1,t) .
\end{equation}
When time parameters $(t,s) \in \K^2$ are fixed, we 
define
$U_{(t,s)}:= 
U^{[1]}_{(t,s)} :=U^\sett{1}_{(t,s)}$ by (\ref{eqn:U_ts}), and the {\em 
(restricted) anchor}
\begin{equation}
\Upsilon_{(t,s)} := \Ups_{(t,s)}^\sett{1} :  U^{[1]}_{(t,s)} \to U \times U
\end{equation}
is given by restricting the map $\Ups$ defined above, i.e., it is given 
by (\ref{eqn:anchor}).
Direct computation shows that $\Ups_{(t,s)}$
 is invertible iff $s-t$ is invertible in $\K$, with inverse given by
(\ref{eqn:anchor-inverse}). 
Note that $(U^{[1]}_{(t,s)},V^2)$ is again a locally linear set, and hence the construction can be iterated, with some new
parameter $(t_2,s_2)$, and so on. Explicit formulae, describing this,
 will be given later (restricted iteration, Def.\ \ref{def:n-th}). 

\subsection{The topological setting}
In the remainder of this section we assume that $\K$ is a {\em good topological ring} (i.e., a topological ring whose unit group $\K^\times$ is open and dense,
and inversion is a continuous map),  that all $\K$-modules are topological modules, and that all locally linear sets $(U,V)$, $(U',V'), \ldots$ are
{\em open} inclusions. 

\begin{definition}
We say that $f:U \to V'$ is {\em of class $C_1^\K$} if the {\em slope} given by (\ref{eqn:slope}) extends to a {\em continuous} map
$f^{[1]} : U^{[1]} \to V'$.
We then define, for all $(x,v) \in U \times V$,
$$
df(x)v:= \partial_v f(x) := f^{[1]}(x,v;0,0).
$$ 
\end{definition}

\begin{remark}
Letting $s=0$, the preceding definition
 clearly implies that $f$ is of class $C^1_\K$
in the sense of \cite{BGN04} or \cite{Be08}. 
Conversely, the map denoted here by $f^{[1]}$ can be expressed by the one denoted $f^{[1]}$ in loc.\ cit., and hence the $C^1_\K$-notions
used there are equivalent to the one given above.
We call the calculus obtained by restricting to $s=0$ {\em target calculus}.
Recall from \cite{BGN04} that, in the real or complex 
finite dimensional
case this definition is equivalent to all usual ones, and in the 
infinite dimensional locally convex case it is equivalent to Keller's definition
of differentiability. 
 \end{remark}

\begin{lemma} \label{la:C1}
For a map $f:U \to U'$, the following are equivalent:
 \begin{enumerate}
 \item
 $f$ is $C^1_\K$,
 \item
 for all $(t,s) \in \K^2$, there exists a (unique)
 map $f_{(t,s)}=f^\sett{1}_{(t,s)}:U_{(t,s)} \to U_{(t,s)}'$, such that
 \begin{enumerate}
 \item
  the map $U^{[1]} \to (U')^{[1]}$, 
 $(x,v;t,s) \mapsto f_{(t,s)}(x,v)$  is continuous,
 \item for all $(t,s) \in \K^2$, 
 $$
 \Upsilon_{(t,s)} \circ f^\sett{1}_{(t,s)} = (f \times f) \circ \Upsilon_{(t,s)}: \qquad
 \begin{matrix} U_{(t,s)} &
 {\buildrel{f_{(t,s)}} \over  \longrightarrow} & U_{(t,s)}'  \\
\Ups \downarrow \phantom{\Ups} & & \phantom{\Ups}\downarrow \Ups \\
 U\times U & {\buildrel{f \times f} \over  \longrightarrow}  & U' \times U'
 \end{matrix}
 $$
 \end{enumerate}
\end{enumerate}
\end{lemma}
 
\begin{proof}
As we have already seen,
when $t-s$ is invertible in $\K$, then $f_{(t,s)}$ is necessarily given by (\ref{eqn:f_ts}). Since its second component is the slope $f^{[1]}$,
existence of $f_{(t,s)}$, jointly continuous in $(x,v;t,s)$, implies existence of a continuous extension of the slope, whence
(2) $\Rightarrow$ (1). To prove the converse, assume (1)
and write $f_{(t,s)}(x,v) = (w_0,w_1)$ with $(w_0,w_1)$ given by
(\ref{eqn:f_ts}).
Assumption (1) means that $w_1 = w_1(x,v;t,s)$  admits a continuous extension.
 Let us show that $w_0 = w_0(x,v;t,s)$  also admits a continuous extension.
 To see this,
let $x_0 := f(x+s v)$ and $x_1:= f(x+tv)$.
Then $x_0 = w_0 + s w_1$, $x_1 = w_0 + t w_1$, whence 
$$
w_0 = x_1 - t w_1  = f(x+tv) -  t f^{[1]}(x,v;t,s) ,
$$
showing that $w_0(x,v)$
 extends continuously for all $(t,s)$ if so does $f^{[1]}(x,v;t,s)$.
\end{proof}

\begin{example}
If $f:V \to V'$ is {\em linear} and
 continuous, then direct computation using (\ref{eqn:f_ts})
shows that $f_{(t,s)}(v_0,v_1) = (f(v_0),f(v_1))$, so $f$ is $C^1_\K$.
\end{example}

\begin{remark}
Letting $v_1 = 0$  in (\ref{eqn:f_ts}), we always get
$f_{(t,s)}(v_0,0)= (f(v_0),0)$. In diagrammatic form, the map $f$ itself
imbeds into $f_{(t,s)}$: we define  the imbedding
\begin{equation}
\iota_{(t,s)}:
U \to U_{(t,s)}, \quad v_0 \mapsto (v_0,0) 
\end{equation}
then the computation just mentioned shows 
 that $f_{(t,s)} \circ \iota_{(t,s)} = \iota_{(t,s)} \circ f$:
\begin{equation}\label{eqn:imbedding}
\begin{matrix} U_{(t,s)} &
 {\buildrel{f_{(t,s)}} \over  \longrightarrow} & U_{(t,s)}'  \\
\iota  \uparrow \phantom{\Ups} & & \phantom{\Ups}\uparrow \iota  \\
 U  & {\buildrel{f} \over  \longrightarrow}  & U 
 \end{matrix}
\end{equation}
Note that
$\Upsilon \circ \iota$ is the diagonal imbedding
$\Delta:U \to U \times U$, $x \mapsto (x,x)$. 
\end{remark}

In this setting, the usual rules of first order calculus hold (chain rule, product
rule, linearity of first differential) -- 
for a systematic exposition we refer to \cite{BGN04, Be08, Be11}.
Most important for our purposes is the Chain Rule, which we write in
functorial form
\begin{equation}
\forall (t,s) \in \K^2 : \quad
(g \circ f)_{(t,s)} = g_{(t,s)} \circ f_{(t,s)}.
\end{equation}
This follows easily from Lemma \ref{la:C1}: for invertible $t-s$, we have
functoriality $ (g\times g)\circ (f\times f) = (g \circ f)\times (g \circ f)$, and for
general $(t,s)$, it  follows ``by density''.

\subsection{Full versus restricted iteration}
Higher order differentiability is defined by iterating first order differentiability. 
However, there are various
ways of doing so, and it is important to distinguish them.
In \cite{BGN04}, $f$ is defined  to be of class $C^2_\K$ if it is $C^1$ and if
$f^{[1]}$ also is $C^1$, so that we can define $f^{[2]}:= (f^{[1]})^{[1]}$, 
etc.:

\begin{definition}[Full iteration]
We say that $f$ is {\em of class $C^n_\K$} if:
$f$  is of class $C^{1}_\K$, and 
$f^{[1]}$ is of class $C^{n-1}_\K$. In this  case we let
$f^{[n]}:= (f^{[1]})^{[n-1]}$.
\end{definition}

\begin{remark}\label{rk:full}
In the real or complex finite dimensional
case this is equivalent to the usual definitions (see
\cite{BGN04, Be11}). However, since 
full iteration repeats the procedure for all variables together,
 the number of variables exploses, and it
is hard to get control over the structure of the maps $f^{[n]}$
(see \cite{Be15b}).
To reduce the number of variables, in   {\em restricted iteration}
we consider in each step time variables to be frozen, and
take difference quotients only with respect to space variables. 
\end{remark}

\begin{notation}
For each $k\in \N$, we denote by an upper index $\sett{k}$ a copy of the
procedure $\{ 1 \}$ that has been defined above. An upper index
$\sett{i ,j}$ ($i<j$)  indicates a double application of the procedure 
(first $\{ i \}$, then $\{ j \}$), etc. 
E.g., an upper index $\sfn :=\sett{1,\ldots,n}$
indicates that we first apply $\{ 1\}$, then $\{2\}$, etc., and finally $\{ n \}$. 

\ssk
To abbreviate, in the sequel, we let 
$(\ttt,\sss)=(t_1,\ldots,t_n;s_1,\ldots,s_n) \in \K^{2n}$.
\end{notation}

\begin{definition}[Restricted iterated domain]
For $U \subset V$, define $U^\sfn_{(\ttt,\sss)} \subset V^\sfn_{(\ttt,\sss)}$ by
$$
U^\sfn_{(\ttt,\sss)} : =
U^\sett{1,\ldots,n}_{(\ttt,\sss)} := 
(\ldots (U_{t_1,s_1}^\sett{1})_{(t_2,s_2)}^\sett{2}) \ldots )_{(t_n,s_n)}^\sett{n} 
=
( U^\sett{1}_{(t_1,s_1)} )^\sett{2,\ldots,n}_{(t_2,\ldots,t_n,s_2,\ldots,s_n)} .
$$
Note that $V_{(t_i,s_i)} \cong V^2$, so
$V^\sfn_{(\ttt,\sss)} \cong V^{(2^n)}$.
\end{definition}

\begin{definition}[Restricted iteration]\label{def:n-th} 
A map $f:U \to U'$ is called {\em of class $C_{\K,n}$} if: it is of class
$C_\K^1$, and, for all $(t_1,s_1)\in \K^2$, the map
$f^\sett{1}_{(t_1,s_1)}$ is of class $C_{\K,n-1}$. In this case we define
inductively 
$$
f^\sfn_{(\ttt,\sss)} :=
( f^\sett{1}_{(t_1,s_1)} )^\sett{2,\ldots,n}_{(t_2,\ldots,t_n,s_2,\ldots,s_n)}
= (\ldots (f_{t_1,s_1}^\sett{1})_{(t_2,s_2)}^\sett{2}) \ldots )_{(t_n,s_n)}^\sett{n}:
U^\sfn_{(\ttt,\sss)} \to (U')^\sfn_{(\ttt,\sss)}.
$$
We also require that $f^\sfn_{(\ttt,\sss)}$ be jointly continuous both in
space and in time variables.
\end{definition}

\begin{theorem}\label{th:Keller}
When $\K = \R$ or $\mathbb C$, and $V$ is a locally convex topological vector space,
then the conditions $C^n_\K$ and $C_{\K,n}$ are both equivalent to the 
usual (Keller's) definition of $C^n$-maps.
\end{theorem}

\begin{proof}
As already mentioned, $C^n_\K$ clearly implies $C_{\K,n}$, and
equivalence of $C^n_\K$ with Keller's definition has been
proved in \cite{BGN04}.
On the other hand,  $C_{\K,n}$ obviously implies Keller's
$C^n$-definition, which arises simply by taking $(\ttt,\sss)=(0,\ldots,0)$
in the $C_{\K,n}$-condition. Thus all three conditions are equivalent.
\end{proof}

\begin{remark} For general $\K$,  properties $C^n_\K$ and
$C_{\K,n}$ cease te be equivalent: in {\em positive characteristic},
condition $C^n_\K$ appears to be strictly stronger than $C_{\K,n}$
 (cf.\ the proof of the general
Taylor formula in \cite{BGN04, Be11}, which really uses {\em full} iteration;
concerning this item, cf.\ also  \cite{Be13}). It would be interesting to
have a criterion allowing to decide when $C^n_\K$ and
$C_{n,\K}$ are equivalent.
\end{remark}

\begin{definition}\label{def:anchor-n}
For all $(\ttt,\sss) \in \K^{2n}$, the {\em $n$-th order anchor} of $U \subset V$
is defined as follows:
for all locally linear sets $(U,V),(U',V')$, we consider the map
$$
(U \times U')_{(t,s)} \to U_{(t,s)} \times U_{(t,s)}',\quad
((v_0,v_0'),(v_1,v_1')) \mapsto ((v_0,v_1),(v_0',v_1'))
$$
as identification. Under such identifications, the map $\Ups:=\Ups^\sfn_{(\ttt,\sss)} :=$
$$
( \Ups^\sett{1}_{(t_1,s_1)} )^\sett{2,\ldots,n}_{(t_2,\ldots,t_n,s_2,\ldots,s_n)}:
U^\sfn_{(\ttt,\sss)} \to 
( U^\sett{1}_{(t_1,s_1)} )^\sett{2,\ldots,n}_{(t_2,\ldots,t_n,s_2,\ldots,s_n)} 
\times 
( U^\sett{1}_{(t_1,s_1)} )^\sett{2,\ldots,n}_{(t_2,\ldots,t_n,s_2,\ldots,s_n)} 
$$
inductively
gives rise to a map
$\Ups^\sfn_{(\ttt,\sss)} :  U^\sfn_{(\ttt,\sss)} \to U^{2^n}$
which we call the {\em $n$-fold anchor}.
\end{definition}

\begin{remark}
In order to fully formalize this definition, we need an explicit
labelling of the $2^n$ copies of $U$ in $U^{2^n}$. 
For the moment, this is not needed, and will be taken up later
(Def.\  \ref{def:anchor-nbis}). 
Let us, however, give the result for $n=2$:
space variables have labels
$0,1,2,12$ corresponding to the subsets of $\sett{1,2}$,
so we write $\vvv = (v_0,v_1,v_2,v_{12}) \in U^\sett{1,2}_{(t_1,t_2,s_1,s_2)}$.
Then iteration shows that
the linear map $\Ups$ is given by the (block) matrix 
(Kronecker
product of two first-order anchors)
\begin{equation}\label{eqn:second-anchor}
 \begin{pmatrix} 
1 & s_1 \\ 1 & t_1 
\end{pmatrix} \otimes
 \begin{pmatrix} 
1 & s_2 \\ 1 & t_2 
\end{pmatrix} = 
\begin{pmatrix} 
1 & s_1 & s_2 & s_1 s_2 \\
1 & t_1 & s_2 &  t_1 s_2 \\
1 &  s_1 & t_2 &  s_1 t_2 \\
1 & t_1  & t_2 & t_1 t_2 
\end{pmatrix}  ,
\end{equation}
so we have four ``evaluation points'' given by the four lines of the
(block) matrix:
\begin{equation}\label{eqn:evaluation-points}
\begin{matrix}
\Ups_{\emptyset}(\vvv) & = &
v_\emptyset +s_1 v_1 + s_2 v_2 + s_1 s_2 v_{12} , \\
\Ups_{1}(\vvv) & = &
v_\emptyset + t_1 v_1 + s_2 v_2 + t_1 s_2 v_{12}, \\
\Ups_{2}(\vvv) & = &
v_\emptyset + s_1 v_1 + t_2 v_2 + s_1 t_2 v_{12},  \\
\Ups_{12}(\vvv) & = & 
v_\emptyset + t_1 v_1 + t_2 v_2 + t_1 t_2 v_{12}.
\end{matrix}
\end{equation}
The inverse matrix of (\ref{eqn:second-anchor}) is the Kronecker
product of the inverses of the respective first order anchors (when
these are invertible): it is given by 
\begin{equation}\label{eqn:second-anchor-inverse}
\frac{1}{t_1-s_1 }
 \begin{pmatrix} 
t_1 & - s_1 \\ -1 & 1 
\end{pmatrix} \otimes
\frac{1}{t_2-s_2}
 \begin{pmatrix} 
t_2 & -s_2 \\ -1 & 1 
\end{pmatrix} =
\frac{1}{(\ttt - \sss)_\sftwo }
\begin{pmatrix} 
t_1 t_2 & -s_1 t_2 &- t_1 s_2 & s_1 s_2 \\
-t_2 & t_2 & s_2 & -s_2 \\
-t_1 & s_1 & t_1 & -s_1 \\
1 & -1 & -1 & 1 
\end{pmatrix}  
\end{equation}
 where $ (\ttt - \sss)_\sftwo  : = (t_1- s_1)(t_2 - s_2)$.
For the general case, see Theorem \ref{th:Anchor-inverse}. 
\end{remark}

\begin{theorem} \label{th:Cn}
For a map $f:U \to U'$, the following are equivalent:
 \begin{enumerate}
 \item
 $f$ is $C_{\K,n}$,
 \item
 for all $(\ttt,\sss) \in \K^{2n}$, there exists a (unique) map 
 $f^\sfn_{(\ttt,\sss)}:U_{(\ttt,\sss)}^\sfn  \to (U')_{(\ttt,\sss)}^\sfn$, such that
 \begin{enumerate}
 \item
 $f^\sfn_{(\ttt,\sss)}(\vvv)$ is jointly continuous in space 
 and time variables $(\vvv;\ttt,\sss)$,
 \item
 for all $(\ttt,\sss) \in \K^{2n}$,  
 $\Upsilon^\sfn_{(\ttt,\sss)} \circ f^\sfn_{(\ttt,\sss)} =  f^{2^n}  \circ \Upsilon^\sfn_{(\ttt,\sss)}$:
$$
 \begin{matrix} U^\sfn_{(\ttt,\sss)} &
 {\buildrel{f^\sfn_{(\ttt,\sss)}} \over  \longrightarrow} & (U')^\sfn_{(\ttt,\sss)}  \\
\Ups^\sfn_{\ttt,\sss} \downarrow \phantom{\Ups} & & \phantom{\Ups}\downarrow \Ups^\sfn_{(\ttt,\sss)} \\
 U^{2^n} & {\buildrel{f^{2^n}} \over  \longrightarrow}  & (U')^{2^n}.
 \end{matrix}
 $$
 \end{enumerate}
\end{enumerate}
The map $f^\sfn_{(\ttt,\sss)}$ depends functorially on $f$: 
$(f \circ g)^\sfn_{(\ttt,\sss)} = f^\sfn_{(\ttt,\sss)} \circ g^\sfn_{(\ttt,\sss)}$
{\em (Chain Rule)}.
\end{theorem}

\begin{proof}
By induction: for $n=1$, this is Lemma \ref{la:C1}.
Assume the claim holds on level $n-1$
and apply it to $f$ replaced by $f^\sett{1}_{(t_1,s_1)}$.
From the inductive definitions, it follows readily that the properties
are again equivalent on level $n$.  The (higher order) Chain Rule now also follows
by induction. 
\end{proof}

\begin{example}
Using Formula (\ref{eqn:second-anchor-inverse}), let us give explicit formulae for
$n=2$: 

\msk
$\quad
f^\sftwo_{(t_1,t_2,s_1,s_2)}(\vvv) = \Ups^{-1} 
\bigl( f(\Ups_\emptyset(\vvv)), f(\Ups_1(\vvv)),f(\Ups_2(\vvv)),f(\Ups_{12}(\vvv))
\bigr)  $
\begin{equation} 
= \frac{1}{(\ttt-\sss)_\bftwo}
\begin{pmatrix}
t_1 t_2 f(\Ups_\emptyset \vvv) - s_1 t_2 f(\Ups_1 \vvv) - t_1 s_2 f(\Ups_2 \vvv)
+ s_1 s_2 f(\Ups_{12}\vvv) 
\\
-t_2  f(\Ups_\emptyset \vvv) + t_2 f(\Ups_1 \vvv) + s_2 f(\Ups_2 \vvv)
- s_2 f(\Ups_{12}\vvv) 
\\
-t_1 f(\Ups_\emptyset \vvv) + s_1 f(\Ups_1 \vvv) + t_1 f(\Ups_2 \vvv)
- s_1 f(\Ups_{12}\vvv) 
\\
f(\Ups_\emptyset \vvv) - f(\Ups_1 \vvv) -  f(\Ups_2 \vvv)
+  f(\Ups_{12}\vvv) 
\end{pmatrix}
\end{equation}
Since $(\ttt - \sss)_\bftwo=(t_1-s_1)(t_2-s_2)$, the first term is in fact an
affine combination of values of $f$ at the four evaluation points, whereas
the other three terms are ``zero-sum combinations'' of these values, and
hence correspond to ``true'' difference quotients.
In order to state results at arbitrary order, we need some
notation:
\end{example}

\subsection{Hypercube notation, and formula for higher order slopes}

\begin{definition}
We call {\em $n$-hypercube} the power set  $\cP(\sfn)= \cP(\sett{1,\ldots,n})$.
It serves as index set for space variables,
which we write in the form 
$\vvv = (v_A)_{A \in \cP(\sfn)}$. 
Recall that $\cP(\sfn)$ is a semigroup for union $\cup$ and intersection
$\cap$, and a group with respect to 
the {\em symmetric difference} 
$$
A \Delta B = (A \cup B) \setminus (A \cap B) = (A \cap B^c) \cup (B \cap A^c),
$$
where  $A^c = \sfn \setminus A$ is the complement of $A$ in $\sfn$.
Recall also that $A^c \Delta B^c = A \Delta B$, and that
$A \Delta B^c = (A \Delta B)^c = A^c \Delta B$, whence
$\vert A \Delta B^c \vert = n - \vert A \Delta B \vert$. 
\end{definition}

\begin{definition}\label{def:t-notation}
For all $\ttt,\sss \in \K^n$ and $A \in \cP(\sfn)$, we let
$\ttt_\emptyset = 1 = \sss_\emptyset$, and
$$
\ttt_A =\prod_{k\in A} t_k, 
\qquad 
\sss_A= \prod_{k\in A} s_k ,
\qquad
(\ttt - \sss)_A = \prod_{k\in A} (t_k - s_k) .
$$
Call $(\ttt,\sss)$ {\em regular}, or {\em finite}, if,
$\forall i =1,\ldots,n: (t_i - s_i)\in \K^\times$, and {\em singular} if
$\forall i =1,\ldots,n: (t_i - s_i)\notin \K^\times$, and {\em mixed} else.
\end{definition}

\begin{theorem} \label{th:slope-n}
Let $f:U \to U'$ be of class $C_{\K,n}$. Then, for all regular $(\ttt,\sss) \in \K^{2n}$,
 and all
$B \in \cP(\sfn)$,
the component $(f_{(\ttt,\sss)}^\sfn(\vvv))_B$ is given by 
\begin{align*}
(f_{(\ttt,\sss)}^\sfn(\vvv))_B  & 
=\frac{1}{(\ttt - \sss)_\sfn}  
\sum_{A \in \cP(\sfn)} 
(-1)^{\vert A \Delta B \vert} \sss_{B^c\cap A}\ttt_{B^c \cap A^c} \, 
 f\bigl( \sum_{C \in \cP(\sfn)} \sss_{C \cap A^c} \ttt_{C \cap A} v_C  \bigr).
\end{align*} 
\end{theorem}

The proof will be given in Subsection \ref{ssec:slope-n}.
For $B = \emptyset$, the component is an affine combination of values
of $f$ at the $2^n$ evaluation points, and for all other components it is again
a ``zero sum combination''. 

\subsection{Categories of locally linear sets and $C_{\K,n}$-maps}
To summarize, we have defined a {\em category of locally linear sets and their
$C_{\K,n}$-morphisms}:

\begin{definition}\label{def:LL}
We denote by $\LL_{\K,n}$ the category whose objects are pairs $(U,V)$, where $V$ is a topological $\K$-module and
$U \subset V$ a non-empty open subset, and morphisms are
$C_{\K,n}$-maps $f:U \to U'$. 
(For $n=0$, morphisms are continuous maps, and for $n=\infty$, these are
maps that are $C_{\K,n}$ for all $n\in \N$.) 
\end{definition}

\begin{definition}\label{def:tg-functor}
For $m \geq n$ and
$(\ttt,\sss)\in \K^{2n}$, the {\em $(n;\ttt,\sss)$-tangent functor} is the 
functor from $\LL_{\K,m}$ to $\LL_{\K,m-n}$ given by
$(U,V) \mapsto (U^\sfn_{(\ttt,\sss)}, V^\sfn_{(\ttt,\sss)}) $ and
$f \mapsto f^\sfn_{(\ttt,\sss)}$.
\end{definition}

\begin{remark}[Manifolds]\label{rk:manifold1}
By the usual glueing procedures, one may now define {\em $C_{\K,n}$-manifolds
over $\K$}, modelled on locally linear sets -- since these methods are 
independent of the particular form of differential calculus, we do not wish
to go here into details (see \cite{Be16} for a formulation of such
principles, adapted to most general contexts). 
The $(n;\ttt,\sss)$-tangent functor then carries over to manifolds :
for every $\K$-smooth manifold $M$ we have a ``generalized higher order
 tangent  bundle'' $M^\sfn_{(\ttt,\sss)}$, depending functorially on $M$,
and coming with an anchor map  $M^\sfn_{(\ttt,\sss)} \to M^{2^n}$.
\end{remark}

\section{The rings of calculus: tangent algebras}\label{sec:algebraic}

Our next aim is to understand the $(n;\ttt,\sss)$-tangent functor as a
{\em functor of scalar extension}, from $\K$ to a ring denoted by
$\K^\sfn_{(\ttt,\sss)}$, and which we shall define next.

\subsection{The scaloid, and the algebras $\K_{(\ttt,\sss)}^\sfn$.}
The scaloid is the index set that will be used in the following construction of {\em tangent algebras}:

\begin{definition}\label{def:scaloid}
We call  {\em scaloid}  the {\em free monoid
over $\K^2$}, that is, 
the disjoint union over $n\in \N_0$ of all $\K^{2n}$:
$$
\Scal := \Scal_\K := \coprod_{n\in \N_0} \K^{2n} 
$$
(in the following, we write $(\ttt,\sss)$ with $\ttt,\sss \in \K^n$
for elements of $\K^{2n}$), together with its monoid structure given 
by juxtaposition, and denoted by
$$
(\ttt,\sss) \oplus (\ttt',\sss') = 
(t_1,\ldots,t_n,t_1',\ldots,t_m'; s_1,\ldots,s_n,s_1',\ldots,s_m')
= (\ttt\oplus \ttt' , \sss \oplus \sss')\  .
$$
 \end{definition}

We denote by $\K[X_1,\ldots,X_n]$ the algebra of polynomials
in $n$ variables with coefficients in $\K$.
It can be defined inductively by using the isomorphisms,
where $\otimes_\K$ denotes the tensor product of two
associative $\K$-algebas,
\begin{equation}\label{eqn:poly-2}
\K[X_1,X_2] \cong (\K[X_1])[X_2] \cong \K[X_1] \otimes_\K \K[X_2],
\end{equation}
 so, by induction, we have an iterated tensor product of algebras
\begin{equation}\label{eqn:poly-n}
\K[X_1,\ldots,X_n] \cong
\K[X_1] \otimes_\K \ldots \otimes_\K \K[X_n].
\end{equation}

\begin{definition}\label{def:tangent_algebra}
For $(\ttt,\sss) \in \K^{2n}$, we define the 
{\em $(\ttt,\sss)$-tangent algebra}
$$
\K_{(\ttt,\sss)}^\sfn :=
\K[X_1,\ldots,X_n] / ((X_i - t_i)(X_i-s_i) , i=1,\ldots,n)  
$$
(quotient by the ideal $I_{(\ttt,\sss)}$ generated by all
$(X_i - t_i)(X_i-s_i) , i=1,\ldots,n$).
\end{definition}

\begin{lemma}\label{la:algeba-induction}
The algebra $\K_{(\ttt,\sss)}^\sfn$ is a free $\K$-module of dimension
$2^n$, having a canonical basis indexed by elements $A$ of
the $n$-cube  $\cP(\sfn)$,
$$
e_A := [X^A], \qquad
X^A = \prod_{k\in A} X_k .
$$
It is also isomorphic to an $n$-fold tensor product of first
order tangent algebras
$\K_{(t_i,s_i)}^\sett{i} = \K[X_i] /((X_i-s_i)(X_i-t_i))$: 
$$
\K_{(\ttt,\sss)}^\sfn =\K_{(t_1,s_1)}^\sett{1} \otimes \ldots \otimes \K_{(t_n,s_n)}^\sett{n} .
$$
\end{lemma}

\begin{proof}
For $n=1$, the claim is obviously true: a polynomial algebra
$\K[X]$ quotiented by the ideal generated by a polynomial
of degree $2$ is of dimension $2$, with $\K$-basis the
classes $[1]$ and $[X]$.
For $n>1$, 
the claim follows by induction using (\ref{eqn:poly-2}).
\end{proof}

\begin{theorem}
Assume $\K$ is a good topological ring.
Then the structure maps $+$ and $\cdot$ of the ring $\K$ are of class
$C_{\K,\infty}$, and applying $n$-fold restricted iteration with parameters
$(\ttt,\sss)$ yields
a good topological ring which is canonically isomorphic to
$\K_{(\ttt,\sss)}^\sfn$ (whence in particular is a free $\K$-module of dimension $2^n$)
\end{theorem}

\begin{proof}
The structure maps are continuous and (bi)-linear, hence smooth
(both in the full and restricted sense, cf.\ \cite{BGN04}). By functoriality, and applying, concerning Cartesian products, the convention from 
Def.\ \ref{def:anchor-n}, rings are transformed by the iterated functors 
into rings. 
We have to show that the ring structure on the underlying set of
$\K_{(\ttt,\sss)}^\sfn$ is precisely the one defined above.
For $n=1$ and regular $(t_1,s_1)=(t,s)$, this follows from the
explicit formulae for difference calculus :
slightly more general, given a
 bilinear continuous map
$\beta:V \times W \to  Y$, thought of 
as a ``product'', so let us write $v \bullet w:=\beta(v,w)$, we compute
$$
\beta_{(t,s)}^\sett{1} : V_{(t,s)}\times W_{(t,s)} \to Y_{(t,s)},\quad
( \begin{pmatrix} v_0\\ v_1 \end{pmatrix} ,
\begin{pmatrix} w_0\\ w_1 \end{pmatrix} ) 
 \mapsto
 \begin{pmatrix} v_0\\ v_1 \end{pmatrix} \bullet_{(t,s)}^\sett{1}
\begin{pmatrix} w_0\\ w_1 \end{pmatrix}
$$ 
which by an explicit computation using Formula (\ref{eqn:f_ts}) is given by
\begin{equation}\label{eqn:product}
\begin{pmatrix} v_0 \\  v_1 \end{pmatrix}  \bullet^\sett{1}_{(t,s)}
\begin{pmatrix} w_0 \\ w_1 \end{pmatrix} 
= \begin{pmatrix} v_0 \bullet w_0 - st \, v_1 \bullet
w_1  \\ v_0 \bullet w_1 + v_1 \bullet w_0 + (s+t) v_1 \bullet w_1 \end{pmatrix} .
\end{equation}
Now, decomposing the product of $\K_{(t,s)}^\sett{1}$ according to
the canonical basis $e_0 =[1], e_1 = [X]$, we get exactly the same
formula, whence the claim for $n=1$ and regular $(t,s)$.
By density, the claim follows for all $(t,s)$, and by straightforward 
induction, using Lemma \ref{la:algeba-induction}, it now follows for all elements $(\ttt,\sss)$
of the scaloid.
Finally, by general argments (\cite{Be08, Be11}), the ring $\K_{(\ttt,\sss)}^\sfn$
is again ``good''.
\end{proof}

By exactly the same arguments we see also that the structure maps
$V \times V \to V$ and $\K\times V \to V$ of a topological $\K$-module
are smooth, and give by restricted iteration rise to the corresponding
structure maps of the scalar-extended module
$V_{(\ttt,\sss)}^\sfn = V \otimes_\K \K_{(\ttt,\sss)}^\sfn$ ; also,
if $f:V \to V'$ is {\em linear}, then
$f_{\ttt,\sss}^\sfn$ coincides with the algebraic scalar extension
$f \otimes \id_{\K_{(\ttt,\sss)}^\sfn}$.

\subsection{Source and target}\label{ssec:alphabeta}
Evaluation of a class $[P] \in \K[X]/((X-s)(X-t))$ at elements $x\in \K$
 is in general not
well-defined, but it is so for $x=s$ and $x=t$. Thus we get two
algebra morphisms $\alpha,\beta : \K_{(t,s)}^\sett{1}\to \K$, 
called {\em source} and {\em target}
\begin{equation}
\alpha([P]) = P(s), \qquad 
\beta([P]) = P(t).
\end{equation}
(Note that $\alpha$ is coupled with $s$ and $\beta$ with $t$, so the
order of $(s,t)$ matters.) 
With respect to the basis $e_0 = [1]$, $e_1=[X]$, we have
$\alpha(v_0 + v_1 e_1) = v_0 + sv_1$,
$\beta(v_0 + v_1 e_2) = v_0 + tv_1 $, which is in keeping with the definitions
in Subsection \ref{ssec:anchor}.
In Appendix \ref{app:algebras} we describe the structure of  $\K_{(t,s)}^\sett{1}$
in an intrinsic way, via $\alpha$ and $\beta$; this may be useful for a further
structure theory, but is not directly needed
in the sequel.

\subsection{The anchor}
Putting source and target together, the {\em first order anchor} is the
algebra morphism defined by
$$
\Ups_{(t,s)}^\sett{1} : \K_{(t,s)} \to \K \times \K , \quad
[P] \mapsto (\alpha(P),\beta(P)) = (P(s),P(t)).
$$

\begin{lemma}
The first order anchor is an isomorphism if, and only if, $(t,s)$ is
regular, i.e., iff $t-s \in \K^\times$.
\end{lemma}

\begin{proof}
The $\K$-linear map $\Ups_{(t,s)}$ is bijective iff
its determinant $t-s\in \K^\times$, see 
Subsection \ref{ssec:anchor}.
\end{proof}

{\em Higher order anchors} can be defined in two (equivalent)
ways: either by
evaluating (classes of) polynomials in several variables on a 
hypercube of evaluation points, or by tensoring first order
anchors. Here we choose the latter approach. For this,
we need some definitions and conventions:

\begin{definition}[Hypercubic spaces and algebras]\label{def:hypercubic}
Let $N \subset \N$ be a finite subset of cardinal $n$. 
The {\em hypercubic space, based on $N$}, is by definition the free
$\K$-module $\K^{\cP(N)}$ of dimension $2^n$
of functions from $\cP(N)$ to $\K$,
 with its 
canonical basis
$$
E_A = E_A^N : \cP(N)\to \K, \qquad
E_A(A) = 1, \, \, \forall B \not= A : E_A(B) = 0.
$$
A hypercubic space carries several important algebra structures. 
When equipping $\K^{\cP(N)}$ with
its {\em pointwise algebra structure}, i.e., considering it as the algebra of functions
from $\cP(N)$ to $\K$, so that the product of the canonical basis elements is
$$
E_A^N \cdot E_B^N = \delta_{A,B} E_A^N,
$$
we say that $\K^{\cP(N)}$ is the {\em $N$-hypercube algebra}.
When $N=\sfn=\{ 1,\ldots,n\}$, we often omit the upper index
$\sfn$, and just speak of {\em the $n$-hypercube algebra}.
\end{definition}

\begin{remark}\label{rk:P-tensorproduct}
See Appendix \ref{app:hyper-LA} for some basic facts about
 linear algebra on hypercubic spaces (independent of the algebra structure).
For induction procedures, the following remark is useful:
If $N_1$ and $N_2$ are disjoint subsets of $\N$, then
$$
\cP(N_1) \times \cP(N_2) \to \cP(N_1 \sqcup N_2),
\quad (A,B) \mapsto A \cup B
$$
is a bijection, whence we get an isomorphism (of modules, and
of cube-algebras) 
$$
\K^{\cP(N_1)} \otimes \K^{\cP(N_2)} \cong  \K^{\cP(N_1 \sqcup N_2)}.
$$
In particular, by induction, there is a canonical isomorphism
$$
\K^{\cP(\{1,\ldots,n\})} \cong
\K^{\cP(\sett{1})}  \otimes \ldots \otimes \K^{\cP(\sett{n})}.
$$
Note that
the neutral element of $\K^{\cP(N)}$ 
is the function that is $1$ everywhere, that is
$$
1 = \sum_{A \in \cP(N)} E_A^N .
$$
\end{remark}

\begin{definition}\label{def:anchor-nbis}
The {\em $n$-fold anchor} is the tensor product of $n$ copies of 
the first order anchor: it is the
algebra morphism
$$
\Ups_{(\ttt,\sss)}^\sfn : =
\otimes_{i=1}^n \Ups_{(t_i,s_i)}^\sett{i}: 
\K_{(\ttt,\sss)}^\sfn \to \K^{\cP(\sfn)},
$$
where for each $k \in \N$,
$\Ups_{(t_k,s_k)}^\sett{k} : \K^\sett{k}_{(t_k,s_k)} \to \K^{\cP(\sett{k})}$
is a copy of the first order anchor. Thus, by definition, 
$$
\Ups_{(t_k,s_k)}^\sett{k} (e_\emptyset) =
E_\emptyset^k + E_k^k, \qquad
\Ups_{(t_k,s_k)}^\sett{k} (e_k) =
s_k E_\emptyset^k + t_k E_k^k .
$$
\end{definition}

For the categorical approach, it is not strictly necessary to have
an explicit formula for the higher order anchor; however, such a 
formula allows to derive the explicit formula for the higher order
slopes, and thus makes the whole procedure algorithmic and computable.
Recall Formula (\ref{eqn:second-anchor}) for the matrix of the
second order anchor, which is the Kronecker product of two
first-order anchors. 
Note that, when $s_1 = 1 = s_2$, then this matrix is a {\em symmetric matrix},
whereas for $t_1 = 1 = t_2$, this is not the case. 
  Using notation introduced above, we generalize:

\begin{theorem}\label{th:Anchor} 
Fix $n \in \N$, and $(\ttt,\sss)\in \K^{2n}$.
With respect to the bases
$(e_A)_{A\in \cP(\sfn)}$ in its domain  and 
$(E_A)_{A \in \cP(\sfn)}$ in its range, the $n$-fold anchor is given by
$$
\Ups = \Ups^\sfn_{(\ttt,\sss)} =
\sum_{(A,B)\in \cP(\sfn)^2}
 \ttt_{A\cap B} \sss_{A \cap B^c}  \, \,  e_A^* \otimes E_B .
$$
In other terms, it is characterized by the following equivalent
conditions:
\begin{enumerate}
\item 
$\Upsilon(e_A) = \sum_{B \in \cP(\sfn)} \ttt_{A\cap B} \sss_{A \cap B^c} E_B$,
\item
$
\Ups(\sum_{A \in \cP(\sfn)} v_A e_A ) =
\sum_{B \in \cP(\sfn)} \bigl(
\sum_{A \in \cP(\sfn)} \ttt_{A \cap B} \sss_{A \cap B^c} v_A  
\bigr) E_B
$,
\item
 the matrix of $\Upsilon$ with respect to these bases
has coefficients
$$
\Ups_{(B,A)} := E_B^* (\Ups (e_A)) =
\ttt_{A\cap B} \sss_{A \cap B^c} ,
\qquad (A,B) \in \cP(\sfn)^2.
$$
\end{enumerate}
In particular, in the symmetric case $\sss = -\ttt$, we have
$\Ups_{(B,A)} = (-1)^{\vert A \cap B\vert } \sss_A$, so
$$
\Ups = \Ups^\sfn_{(-\sss,\sss)} =
\sum_{A \in \cP(\sfn)} \sss_A \sum_{B \in \cP(\sfn)} 
(-1)^{\vert A \cap B\vert } e_A^* \otimes E_B .
$$
\end{theorem}


\begin{proof}
This is the special case of Theorem \ref{th:hyper-matrix} for
${\bf a} = 1 ={\bf c}$, ${\bf b}=\sss$, ${\bf d} = \ttt$.
\end{proof}

Next, to compute the inverse of the anchor, in the regular case,
recall Formula (\ref{eqn:second-anchor-inverse}) concerning the
case $n=2$.
This generalizes as follows:

\begin{theorem}\label{th:Anchor-inverse} 
Fix $(\ttt,\sss) \in \K^{2n}$. Recall the notation
$(\ttt - \sss)_\sfn = \prod_{k=1}^n (t_k - s_k)$.
 The anchor map $\Ups=\Ups_{(\ttt,\sss)}^\sfn$
 is invertible if, and only if, $(\ttt,\sss)$ is regular, i.e., $t_k - s_k$ is invertible for all
$k=1,\ldots,n$, and then its inverse map is given by the formula
$$
\Ups^{-1} = \frac{1}{(\ttt - \sss)_\sfn}
\sum_{(A,B) \in \cP(\sfn)^2} 
(-1)^{\vert A \Delta B \vert}  \sss_{A^c\cap B}\ttt_{B^c\cap A^c} \, \, 
E_A^* \otimes e_{B} .
$$
Equivalently,
\begin{enumerate}
\item
$
\Upsilon^{-1}(E_A) =
\frac{1}{(\ttt - \sss)_\sfn}
\sum_{B\in \cP(\sfn)} (-1)^{\vert A \Delta B \vert}
\sss_{A^c\cap B}\ttt_{B^c\cap A^c} \, \, e_{B},
$
\item
$
\Ups^{-1} (\sum_{A \in \cP(\sfn)} y_A E_A ) = \frac{1}{(\ttt - \sss)_\sfn}
\sum_{B \in \cP(\sfn)} 
(-1)^{\vert A \Delta B \vert}
y_A \sss_{A^c\cap B}\ttt_{B^c\cap A^c} \, \, e_{B} .
$
\end{enumerate}
In particular, in case $\sss = - \ttt$, we get (using
$(A \Delta B) \sqcup  (A^c \cap B^c) = (A \cap B)^c$)
$$
\Upsilon^{-1}(E_A) =
\frac{1}{(-2)^n \sss_\sfn }
\sum_{B\in \cP(\sfn)}  (-1)^{\vert A \cap B\vert}
\sss_{B^c} \, e_{B}.
$$
\end{theorem}

\begin{proof}
This is a special case of Theorem \ref{th:hyper-inverse}.
\end{proof}

\subsection{The $n$-th order restriced slope map}\label{ssec:slope-n}
Having established the explicit formulae for $\Ups$ and $\Ups^{-1}$, 
we can prove the already anounced formula from Theorem \ref{th:slope-n}
for  $f^\sfn_{(\ttt,\sss)} = \Ups^{-1} \circ f^{\cP(\sfn)} \circ \Ups$
when $(\ttt,\sss)$ is regular.
We decompose $\vvv \in V_{(\ttt,\sss)}^\sfn = V \otimes_\K \K_{(\ttt,\sss)}^\sfn$ in the form $\vvv = \sum_{A\in \cP(\sfn)} v_A e_A$, and
$\Ups(\vvv) = \sum_{A \in \cP(\sfn)} \Ups_A (\vvv) E_A$, with  the $2^n$ {\em evaluation points} 
given by 
$$
\Ups_A (\vvv)=\sum_{C \in \cP(\sfn)} \sss_{C \cap A^c} \ttt_{C \cap A} v_C  .
$$
Then
\begin{align*}
f_{(\ttt,\sss)}^\sfn  & (\sum_{A\in\cP(\sfn)}  v_A e_A )
= \Ups^{-1} \Bigl(\sum_{A \in \cP(\sfn)} 
f \bigl( \Ups_A(\vvv) \bigr) \Bigr)
\\
&= \frac{1}{(\ttt - \sss)_\sfn}
\sum_{B \in \cP(\sfn)}  e_{B} 
\Bigl( \sum_{A \in \cP(\sfn)} 
(-1)^{\vert A \Delta B \vert} \ttt_{A^c\cap B^c}\sss_{B^c \cap A}
 f\bigl( \Ups_A (\vvv)  \bigr)
 \Bigr) 
 \\
&= \frac{1}{(\ttt - \sss)_\sfn}
\sum_{B \in \cP(\sfn)}  e_{B} 
\Bigl( \sum_{A \in \cP(\sfn)} 
(-1)^{\vert A \Delta B \vert} \ttt_{A^c\cap B^c}\sss_{B^c \cap A}
 f\bigl( \sum_{C \in \cP(\sfn)} \ttt_{C \cap A} \sss_{C \cap A^c} v_C  \bigr)
 \Bigr) .
\end{align*}

\subsection{Target calculus, source calculus, and symmetric calculus}\label{ssec:symmetric}
There are three special cases of calculus, as defined here, 
that deserve attention:
\begin{enumerate}
\item
{\em target calculus}, obtained when $\sss = {\bf 0}$,
i.e., $\forall i$, $s_i=0$;
\item
{\em source calculus}, obtained when $\ttt = {\bf 0}$,
\item
{\em symmetric calculus}, obtained when $\sss = -\ttt$,
i.e., $\forall i$, $s_i + t_i = 0$.
\end{enumerate}
In these cases, the range of scaloid parameter reduces to $\K^n$
instead of $\K^{2n}$, and
the relations satisfied by the canonical basis
$(e_A)_{A \in \cP(\sfn)}$ are relatively simple:
\begin{enumerate}
\item
{\em target calculus}, 
$e_i^2 = t_i e_i$, whence $e_A^2 = \ttt_A e_A$ and
$e_A e_B = \ttt_{A \cap B} e_{A \cup B}$,
\item
{\em source calculus}, same, with $\sss$ instead of $\ttt$,
\item
{\em symmetric calculus},
$e_i^2 
= 4 t_i^2$, so
$e_A^2 = 4^{\vert A\vert} \ttt_A^2$,
$\, e_A e_B = 4^{\vert A \cap B\vert} \ttt_{A\cap B}^2 e_{A \Delta B}$.
\end{enumerate}
The ``most singular value'' is in all cases
$\ttt = {\bf 0}=\sss$, whereas the ``unit value'' is 
\begin{enumerate}
\item
{\em target calculus}, ``unit'' $\ttt = {\bf 1} = (1,\ldots,1)$, $\sss = {\bf 0}$,
\item
{\em source calculus}, ``unit''  $\ttt = {\bf 0}$, $\sss = {\bf 1}$,
\item
{\em symmetric calculus}, ``unit'' $\ttt = {\bf 1}, \sss = -{\bf 1} =
(-1,\ldots,-1)$ (another convention would be to divide this by $2$,
if $2$ is invertible in $\K$).
\end{enumerate}
Thus, taking for $(\ttt,\sss)$ the unit value, the algebra $\K_{(\ttt,\sss)}^\sfn$
with its canonical basis,
\begin{enumerate}
\item
in {\em target calculus}, is the {\em semigroup algebra of the monoid
$(\cP(\sfn),\cup)$},
\item
idem in {\em source calculus}, 
\item
in {\em symmetric calculus}, after normalizing by division by $2$, is
the {\em group algebra of the group $(\cP(\sfn),\Delta)$ with group law
given by the symmetric difference $\Delta$}.
\end{enumerate}
In all three cases, the anchor, being a morphism to the multiplicative
algebra of functions on $\cP(\sfn)$, plays the r\^ole of a 
{\em Fourier transform}.
Namely, for $A \in \cP(\sfn)$, the linear form $E_A^*: \K^{\cP(\sfn)} \to \K$
is the $A$-projection, which is a {\em character}, i.e., an algebra morphism into the base ring. 
Thus the $2^n$ components of $\Ups$,
$$
\Ups_A :=  
E_A^* \circ \Ups : \K^\sfn_{(\ttt,\sss)} \to \K, \quad
x \mapsto \sum_{C \in \cP(\sfn)} \sss_{C \cap A^c} \ttt_{C \cap A} x_C
$$
also are characters (for $n=1$, these are just the source and target
projections; for $n\geq 1$, they can be considered as higher order
versions of source and target maps).
For instance, when $\ttt = - \sss$ is constant $\frac{1}{2}$,
then from the explicit formula above we get all $2^n$ characters
 of the group $(\cP(\sfn),\Delta)$ (for $A \in \cP(\sfn)$),
\begin{equation}
\Ups_A=\chi_A : \cP(\sfn) \to \{\pm 1 \}, \quad B \mapsto \chi_A(B) = (-1)^{\vert
A \Delta B\vert}
\end{equation}
Thus the matrix of $\Ups$ is the character table of the abelian group
$(\cP(\sfn),\Delta)$, which is also the matrix of the Fourier transform
when identifying this group with its dual group.

\section{The categorical approach}\label{sec:categorical}

In the preceding section we have described how to define, starting with
a $\K$-smooth function $f$, a family of functions $(f_{(\ttt,\sss)}^\sfn)_{(\ttt,\sss;n)\in \Scal_\K}$, behaving well with tangent algebras, anchors, and their corresponding scalar extensions.
In the present section, we describe an abstract, categorical setting capturing
the main features of these constructions.
The procedure is very much like the classical one, starting from polynomial
functions, to define abstract polynomial rings. 
In general, one cannot recover all abstract polynomials by polynomial functions;
for this we need assumptions on $\K$ (e.g., of topological nature).

\subsection{The small monoidal categories in question}\label{ssec:monoidal}
Let $\bfc$ be a monoid, with ``product'' denoted by $\oplus$ and neutral
element $0$.
It gives rise to a small category that shall also be denoted by $\bfc$:
its objects are elements $t \in \bfc$, and morphisms are given by 
compositions of left- and right multiplications in the monoid, i.e., of the form
$$
t \to t_1 \oplus t \oplus t_2, \quad t,t_1,t_2 \in \bfc.
$$
The monoids we are interested in will all be {\em left and right
 cancellative}, that is,
$t \oplus s = t' \oplus s \Rightarrow t=t'$ and
$t \oplus s = t \oplus s' \Rightarrow s=s'$; thus the
small category $\bfc$ is {\em skeletal}
in the sense of \cite{CWM}, p. 93: two objects are isomorphic iff
they are equal. 
Now, here are the cases we are interested in:
\begin{enumerate}
\item
The monoid $\N_0$ with its usual addition, and neutral element $0$.
\item
Recall from Definition \ref{def:scaloid} that objects of the scaloid
$\Scal_\K$
are elements $(\ttt,\sss)$ of the free monoid over $\K^2$. 
The neutral element is the empty word.
Morphisms are now defined as above. 
\item
The {\em small category of $\K$-tangent algebras} $\Talg_\K$ has objects
the algebras $\K^\sfn_{(\ttt,\sss)}$ defined in Def.\ \ref{def:tangent_algebra}, together with their
label $(\ttt,\sss)$. The monoidal
structure is given by the tensor product of associative $\K$-algebras, which
now serves to define also the morphisms in this category.
The neutral element is $\K$, labelled by the empty word. 
\end{enumerate}

\begin{lemma}\label{la:Talg-Scal}
The small monoidal categories  $\Talg_\K$ and $\Scal_\K$
are isomorphic (in the sense defined in \cite{CWM}, p.\  92):
under this bijection, $\K_{(\ttt,\sss)}^\sfn$ corresponds to $(\ttt,\sss)$.
\end{lemma}

\begin{proof}
By the definitions given above, the map $\Talg_\K \to \Scal_\K$ is
well-defined, its inverse map is $(\ttt,\sss) \mapsto \K^\sfn_{(\ttt,\sss)}$.
As we have seen in Lemma
 \ref{la:algeba-induction}, this bijection then is an isomorphism
of monoids. 
\end{proof}

\begin{lemma}
The ``length'' or ``degree'' map
$\ell : \Scal_\K \to \N_0$, associating to each word its length, is a monoid
morphism, and defines a functor of monoidal categories.
\end{lemma}

\begin{proof}
Obviously, $\ell$ is a morphism, and by routine computation such a morphism
induces a morphism (functor) of the corresponding monoidal categories.
\end{proof}

\subsection{Functor categories}
Next we consider {\em functor categories}. We mostly follow notations and
conventions from  \cite{CWM, MM}. 
Thus, 
we denote by $\SET$ the (large) category of sets and set-maps, and
(following notation from \cite{MM}, p.\ 25)
by $\SET^\bftwo$ the (large) category of {\em anchored sets}, that is, objects
$(M,\gamma,M')$ are maps $\gamma:M \to M'$, where morphisms are 
{\em anchor-compatible pairs
of maps} $\Phi : M\to N$, $\Phi':M' \to N'$, i.e.
$\gamma_N \circ \Phi = \Phi' \circ \gamma_M$.

\ssk
Functors from a category $C$ to a category $B$, together with
their natural transformations, form a {\em functor category}
$\Fn(C,B)=B^C$ (see e.g. \cite{CWM}, II.4, or \cite{MM}). 
Specifically, we are interested in functor categories
$\Fn(\bfc,\SET)=\SET^\bfc$ or $\Fn(\bfc,\SET^2)$, where
$\bfc$ is one of the small monoidal categories mentioned above.
If 
$\ul M : \bfc  \to \SET$
is a functor, then for every object $a \in \bfc$ we write
$M_a := \ul M(a)$ (the set obtained by applying $\ul M$ to $a$), 
and for every morphism $\phi:a \to b$ of $\bfc$, we write
$M_\phi:M_a \to M_b$ for the induced set-map.
Likewise, 
for each natural transformation $\ul f:\ul M \to \ul N$, we write
$f_a:M_a \to N_a$ for the corresponding set-map from $\ul M(a)$ to $\ul N(a)$. 
The compatibility condition then is
$$
\forall \phi:a\to b,  \, \forall \ul f : \quad
N_\phi \circ f_a = f_b \circ M_\phi .
$$
Composition of natural transformations is defined 
 ``pointwise'', i.e.,  for two  laws $\ul f:\ul M \to \ul N$, $g:\ul N \to \ul P$ and all 
objects $a$ of $\bfc$,   we have 
$(\ul g\circ \ul f)_a := g_a \circ f_a : M_a \to P_a$.

\begin{definition}\label{def:evaluation}
For each object $a$ of $\bfc$,  {\em evaluation at level $a$}, defined by
$$
{\ev}_a : \ul M \mapsto M_a
, \, \ul f \mapsto f_a,
$$
is a functor from $\Fn(\bfc, \SET)$ to $\SET$.
In particular, when $\bfc$ is monoidal with neutral element $0$, we
call simply {\em evaluation} the evaluation $\ev_0$ at $0$.
\end{definition}

In the following, our concern will be to define (``extension'')
 functors that go in the
direction opposite to $\ev_0 : \Fn(\bfc, \SET) \to \SET$.

\subsection{Cubic extensions of sets.}
For each set $M$ and $n\in \N$, we have a hypercube of sets
$M^{\cP(\sfn)} \cong M^{2^n}$. 
This gives rise to a ``cubic extension functor'':

\begin{lemma}
Let us define
$$
\iota :
\SET \to \SET^{\N_0}, \quad
\begin{matrix}
M \mapsto \ull M &:= & (0 \mapsto M, \, n \mapsto M^{\cP(\sfn)})\\
f \mapsto \ull f &:= & (0 \mapsto f, \, n \mapsto f^{\cP(\sfn)}) 
\end{matrix}.
$$
Then $\ull M : \N_0 \to \SET$ is a functor, and (for $f:M \to N$),
$\ull f: \ull M \to \ull N$ is a natural transformation, and
 $\iota$ is a functor from $\SET$ to $\Fn(\N_0,\SET)$
 such that
$\ev_0 \circ \iota = I_\SET$ is the identity functor on $\SET$.
\end{lemma}

\begin{proof}
The main point is to see that  $\ull M$ is a functor.
Indeed, this
 follows from the identifications
 $(M^A)^B = M^{A \times B}$ together with
$\cP(\mathsf{n+m}) = \cP(\sfn) \times \cP(\mathsf{m})$:
$$
M^{\cP(\mathsf{n+m})} = 
M^{\cP(\sfn) \times \cP(\mathsf{m})} =
(M^{\cP(\sfn)})^{\cP(\mathsf{m})} .
$$
(In particular, for $n = 0$, this means that $M = M_0 \to M^{\cP(\mathsf{m})}$
is the diagonal imbedding: an element
$x \in M$ corresponds to the constant function
$x: \cP(\mathsf{m}) \to M$ having value $x$.)
Next,
the properties of a natural transformation for $\ull f$ are easily checked, 
as are those saying that $\iota$ is a functor. 
Finally, by definition, for the neutral element,
$\ull M_0 = M$, whence $\ev_0 \circ \iota(M)=M$. 
\end{proof}

\begin{definition}
Let us call  {\em cubic set} the realisation $\ull M$ of a set $M$ as a functor
described by the lemma, and denote by $\Cu$ 
the image of $\iota$, the {\em cubic realisation
of the category $\SET$}.
\end{definition}

\subsection{Scalar extensions of modules}
On the category $\Mod_\K$ of $\K$-modules with $\K$-linear maps,
we also have the ``usual'' algebraic
scalar extension functor: 

\begin{lemma}
Let us define
$$
\tau :
\Mod_\K  \to \SET^{\Scal_\K}, \quad
\begin{matrix}
V \mapsto \ul V &:= & (n,\ttt,\sss) \mapsto V_{(\ttt,\sss)}^\sfn = V \otimes_\K \K_{(\ttt,\sss)}^\sfn\\
f \mapsto \ull f &:= & (n,\ttt,\sss) \mapsto f_{(\ttt,\sss)}^\sfn =   f \otimes_\K \id_{\K_{(\ttt,\sss)}^\sfn}
\end{matrix}.
$$
This defines a functor from the category
$\Mod_\K$ to $\Fn(\Scal_\K,\SET)$ such that $\ev_0 \circ \tau$ is the identity
functor on $\Mod_\K$.
\end{lemma}

\begin{proof}
All of this is clear from properties of algebraic scalar extensions, along with
the isomorphism of categories $\Scal_\K \cong \Talg_\K$.
(As in the preceding proof, the main point is that $\ul V$ is a functor.
In the present case, this holds more generally for general ring 
morphisms, and not only those coming from the monoidal structure of
$\Scal_\K \cong \Talg_\K$.) 
\end{proof}

\begin{remark}
Clearly, as morphisms in $\Mod_\K$
one could also use affine maps instead of linear ones.
More generally,
following N.\ Roby \cite{Ro63}, one could
replace linear maps $f$ by polynomial morphisms, corresponding to
``polynomial laws'' as defined in loc.\ cit.
\end{remark}


\subsection{$\K$-space laws}\label{ssec:K-space}
Now we define a functor category $\Space_\K$ of {\em smooth $\K$-space
laws}.
One could do so for each fixed $n\in \N$, defining {\em $\K$-space laws
of class $C^n$}, but it is quicker and clearer to do this for all $n \in \N_0$
together. 

\begin{definition}\label{def:K-space}
Objects of $\Space_\K$ are pairs $(\ul M,\ul \Ups)$, where
$\ul M : \Scal_\K \to \SET$ is a functor and
$\ul \Ups : \ul M \to \ull{M_0}$ is a natural transformation,
and morphisms of $\Space_\K$ are natural 
transformations $\ul f:\ul M \to \ul M'$ commuting with anchors
in the sense that 
$$
\ul \Ups' \circ \ul f = \ull{f_0} \circ \ul \Ups : \, \, \ul M \to \ull{M_0'}.
$$
We require that $\Mod_\K$ is a subcategory of $\Space_\K$, in the sense that
on $\Mod_\K$ the extensions coincide with algebraic scalar extensions coming from the
corresponding ring extensions:
when 
$V$ is a $\K$-module, then 
$\ul \Ups : \ul V \to \ull{V_0}$ is, for each $(n,\ttt,\sss) \in \Scal_\K$, given by the anchor of scalar extensions
$\Ups_{(\ttt,\sss)}^\sfn : V_{(\ttt,\sss)}^\sfn \to V^{\cP(\sfn)}$.
\end{definition}

Equivalently,
a $\K$-space law $(\ul M,\Ups_M)$ could also be defined as a
functor from $\Scal_\K$ to  $\SET^\bftwo$, the category of ``anchored
sets'', satisfying certain properties. 
The present formulation features the anchor as a kind of
``underlying morphism'' of functor
categories
$ \Space_\K \to \Cu \cong \SET$.
At this point,
the situation is quite similar to the one given by abstract polynomials
$P \in \K[X]$, to which we can associate, by evaluation on $\K$,
 an underlying set-map
$\tilde P:\K \to \K$.
In order to define a functor in the other direction, we need assumptions.

\subsection{The topological case}
Let's return to the topological case, and
assume that $\K$ is a good topological ring.
Recall from Definition  \ref{def:LL} the category $\LL_{\K,n}$ of locally linear
sets with $C_{\K,n}$-maps as morphisms
($n\in \N$, or $n = \infty$).

\begin{definition}[Prolongation functor]\label{def:imbedding1}
We define a {\em prolongation
functor} 
$$
\iota : \LL_{\K,\infty} \to
\Fn(\Talg_{\K},\SET)
$$
by associating to an object $(U,V)$ (i.e., $U$ open in a topological
$\K$-module $V$) 
 the functor
$\ul U$ defined by $(\ttt,\sss) \mapsto U_{(\ttt,\sss)}^\sfn$
(Def.\ \ref{def:n-th}), and to a $C_{\K,n}$-map
$f:U \to U'$ the natural transformation defined by restricted iteration
(Def.\ \ref{def:n-th})
$$
\ul f : \qquad
f_\K = f, \quad
f_{\K_{(\ttt,\sss)}^\sfn} = f_{(\ttt,\sss)}^\sfn .
$$
\end{definition}

\begin{lemma}\label{la:imbedding1}
The correspondence $\iota$ defined above 
defines a $\K$-space, it
is indeed a functor,
and 
$$
\ev_0 \circ \iota = \id_{\LL_{\K,\infty}}.
$$
\end{lemma}

\begin{proof}
First of all, $\ul U$ defines a $\K$-space: there is an anchor having the
required properties; it
 is  a functor: 
it is compatible with left and right tensoring, and similarly,
 $C_{\K}$-maps indeed induce natural
transformations.
Finally, the evaluation functor clearly gives us back the original objects
and morphisms, $U_\K = U, f_\K = f$.
\end{proof}

For the moment, the composition $\iota \circ \, \ev_{\K}$ is not even
defined, since the evaluation $\ev_0 (\ul f)$ has no reason to
be a {\em smooth} function. 
Thus our concern will be to define a subcategory where this is the 
case. Since the local linear structure plays a decisive role here,
we restrict our attention to this situation, allowing us to state the
result even as an {\em isomorphism of categories}.

\begin{definition}\label{def:continuous_functor}
Let $\K$ be a good topological ring.
We define the functor category $\mathbf{CSpace}_\K$
of {\em continuous $\K$-space laws} to be the subcategory of
$\Space_\K$ defined as follows: 
 \begin{enumerate}
\item 
categories $\SET$ and $\SET^\bftwo$ are replaced by
$\mathbf{Toplin}$ and $\mathbf{Toplin}^\bftwo$ (open sets in topological 
$\K$-modules, and the corresponding 
 continuous anchors and continuous morphisms, meaning that all
 $\Ups_\bA$ and $f_\bA$ are continuous maps),
\item
 morphisms
 $\ul f$  are moreover {\em jointly continuous in the scaloid}, i.e.: for all 
  locally linear sets $(U,V)$ and morphisms $\ul f$, 
the following map is continuous
(where $V^\sfn_{(\ttt,\sss)} \cong V^{2^n}$ via the $e$-basis, and likewise for
$W^{2^n}$):
 $$
\K^{2n} \times V^{2^n}  \supset \{ (\ttt,\sss;\vvv) \mid \vvv \in U_{(\ttt,\sss)}^\sfn \}
\to W^{2^n}, \quad  
(\ttt,\sss;\vvv) \mapsto f_{(\ttt,\sss)}^\sfn(\vvv) .
$$
\end{enumerate}
\end{definition}

\nin
Inclusions of (non-empty) open sets in topological $\K$-modules, $U \subset V$,  then induce morphisms
$\ul U \to \ul V$, which again will be called ``inclusions''. 
The whole set-up of our theory is designed such that the following result
becomes essentially a tautology:

\begin{theorem}\label{th:imbedding2}
We have two well-defined and mutually inverse functors $\ev$ and $\iota$, 
defining an isomorphism of categories
$$
\LL_{\K,\infty} \cong \mathbf{CSpace}_\K .
$$
In particular, $\iota$ defines a full and faithful imbedding of
$\LL_{\K}$ into a functor category.
\end{theorem}

\begin{proof}
Note that in the present case we can speak of equality of objects
on both sides in question, and hence the notion of ``isomorphism'' of
these categories makes sense (cf.\ \cite{CWM}, p. 92-93).

Starting with a $C_{\K,n}$-function $f$,  it follows from Lemma \ref{la:imbedding1} that $f$ can be identified with evaluation at level
$0$ of the natural transformation $\ul f$ defined by $f$.

To prove the converse,
let $\ul f:\ul U\to \ul U'$ be a continuous morphism of laws. 
We have to show that $\ul f$ is induced by a map of class $C_{\K,n}$;
more precisely, 
we show that the underlying map $f = f_0:U_\K = U \to U' = (U')_\K$
 is of class $C_{\K,n}$, and
that it induces $\ul f$. 
As required in Definition \ref{def:K-space}, the anchor of $V^\sfn_{(\ttt,\sss)}$ is given by 
 $\id_V \otimes \Ups_\K$, and via inclusions, the anchor of $\ul U$ is given by restricting the anchor
 of $\ul V$. 
Since $\ul f$ is a morphism, it commutes with the anchor in the sense that
$$
\Ups \circ f_{\K_{(\ttt,\sss)}^\sfn} = f_{\K^{\cP(\sfn)}}  \circ \Ups .
$$
By the continuity property (2) from Definition \ref{def:continuous_functor}, 
these maps are continuous and jointly
continuous also in $(\ttt,\sss)$,  whence satisfy the condition from Theorem \ref{th:Cn}, showing that the base map
$f = f_\K$ is of class $C_{\K,\infty}$, with the components of $\ul f$ given by
the construction from topological differential calculus; 
thus $f_\K$ induces the  natural transformation $\ul f$.
\end{proof}

\begin{remark}\label{rk:structure}
As usual for ``tautological'' results, the main work lies in the preceding
definitions and auxiliary results.
To make this yet more plain, let's write $G$ for the monoid 
$\Talg_\K \cong \Scal_\K$ (Lemma \ref{la:Talg-Scal})
and $C$ for some subcategory of $\SET^\bftwo$. 
Assuming $C$ to be small, we may consider the set
$C^G$ of all functions from $G$ to $C$.
Clearly, evaluation at the neutral element $o \in G$ defines a map
$\ev_o : C^G \to C$.
The natural candidate for a map in the other direction is 
sending $C$ to the ``constants''
$C \to C^G$, $f \mapsto (g \mapsto f)$.
The problem is that the meaning of ``constants'' has to be carefully
defined in a categorical context. 
\end{remark}

\begin{remark}[Infinitesimal vs.\ local and global] \label{ssec:Dubuc}
A remark on comparison with the case of {\em Weil laws} as defined in
\cite{Be14} is in order here.
Taking for $\bfc_\K$ the category of  {\em Weil algebras}, instead of our tangent algebras,
 we get a formally
quite similar theory. However,  the anchor becomes ``invisible'' (for a Weil algebra,
it degenerates to a single character), 
and one may say that Weil algebras are by nature
{\em infinitesimal objects} (because of the nilpotency condition). Thus the
link with the local and global theory is not encoded by algebra (as in
our approach), and in order to get a {\em well-adapted} model one has to use
more analytic tools 
(so it is not clear how far these can be generalized beyond the case of  real or complex 
base field) -- see \cite{Du79, MR}.
Nevertheless, it might be interesting to look for a category of algebras comprising both
Weil algebras and our tangent algebras -- in order to prepare the ground, in Appendix B,
we describe some algebraic structures that might be useful for such an approach.
\end{remark}

\section{Further directions} \label{sec:further}

With Theorem \ref{th:imbedding2}, we have shown that the functor 
category $\Space_\K$
can be considered as a ``well adapted model'' for general differential
calculus. 
In subsequent work, we will develop the theory further:
on the one hand, comparing with  SDG, we will investigate categorical
questions, on the other hand,  by enriching the
structure of our category of algebras, the theory naturally 
offers links with {\em higher algebra} and with {\em super-calculus}.
We give some short comments on these items.

\subsection{Natural transformations, morphisms}
In the preceding formulation, we have limited morphisms in the monoidal
categories $\Scal_\K$, resp.\ $\Talg_\K$, to the strict minimum necessary
to state the general form of the theory.
However, in differential geometry, other algebra morphisms play a r\^ole
by inducing {\em natural transformations}, as explained by the theory
of Weil-functors (see \cite{KMS93}).
These algebra morphisms appear already on the level of difference calculus:
for instance, the automorhism $\kappa$ (inversion, see Theorem \ref{th:structure})
corresponds to the {\em exchange automorphism} on the level of
$\K^{\cP(\sfone)} \cong \K^2$, inducing a global automorphism on the level
of the functor categories.
Likewise, our monoidal categories are moreover {\em symmetric braided
monoidal}, via the usual braiding $\bA \otimes \bB \cong \bB \otimes \bA$
of associative algebras: again, this gives rise to globally defined morphisms
(Schwarz's Theorem, and the ``canonical flip'' of higher tangent bundles)
which together with the inversions, generate at $n$-th order level an 
automorphism group which is a {\em hyperoctahedral group} (automorphism
group of a hypercube).

\subsection{Groupoids, and higher algebra}\label{ssec:groupoid}
In topological calculus, the extended domains
$U_{(\ttt,\sss)}^\sfn$ carry a natural structure of {\em $n$-fold groupoid}
(by iteration from Item (5) of Theorem \ref{th:structure};
see \cite{Be15, Be15b, Be17}, for the case of target calculus).
This is related to the preceding item: indeed, one can show that
the groupoid structure on $\K_{(\ttt,\sss)}^\sfn$ is internal to the category
of algebras, i.e., all structure maps of the groupoid are algebra morphisms.
However, in order to ``categorify'' this, one needs to enlarge our small
category of 
algebras so that it becomes  stable under more general operations than
just tensor products, such as {\em fiber products}.
This will be taken up in subsequent work.

\subsection{Graded calculus}\label{ssec:super}
We  insist on the importance of the monoidal structure
of the categories $\Talg_\K$ and $\Scal_\K$, with the aim to adapt
the present approach  for giving a  functorial approach to
{\em super}-calculus.
In principle, it seems that the basic structure outlined in Remark
\ref{rk:structure} can be transposed to the monoidal category
of {\em graded algebras and graded tensor products}
generated by $\Ups_{t,s}$.
It remains to understand the precise relation of such a graded
categorical calculus with supercalculus, as it is currently presented.
To do this, on should concentrate on symmetric calculus 
($\ttt = - \sss$), since in this
case the groupoid inversion $\ka$ (which becomes the grading automorphism
of superalgebras) is given by the simple formula
$\ka(v_0 + e v_1)= v_0 - e v_1$ (cf.\ Theorem \ref{th:structure}).



\subsection{Full iteration, and simplicial calculus}
As mentioned in Remark \ref{rk:full}, {\em full iteration} leads to higher order
``tangent maps'' $f^\sett{1,\ldots,n}$ 
having a very complicated structure. In principle,
this structure can also be interpreted in terms of higher groupoids (see
\cite{Be15b}). In this setting, the analog of the tangent algebra category
$\Talg_{\K}$ 
will be some small higher order category, whose structure remains
to be understood yet. 
Restricting again variables to certain subspaces, one can
obtain a sufficiently simple calculus, called {\em simplicial} in \cite{Be13},
and corresponding to the classical concept of {\em divided differences}.
It is certainly possible to put this simplicial calculus into a categorical form,
essentially as done in this work for restricted iteration. 
The advantage should be a better compatibility of calculus with
algebra in {\em positive} characteristic, but the drawback is that the
close link with the tensor product, featured in the present approach, gets lost:
iteration is no longer given by subsequent tensor products. 

\appendix 

\section{Hypercubic linear algebra}\label{app:hyper-LA}

In this appendix, ``linear spaces'' are modules over a commutative
ring $\K$.
Recall Definition \ref{def:hypercubic} of a {\em hypercubic space
based on $N \in \cP(\N)$}.
Changing slightly our viewpoint, every free $\K$-module $V$ with
basis indexed by $\cP(N)$ is isomorphic to $\K^{\cP(N)}$ and hence
will also be called {\em hypercubic space}.

When $f:V \to W$ is linear, for bases $(b_j)_{j \in J}$ in $V$ and
$(c_i)_{i \in I}$ in $W$, we denote by
$f_{i,j} := c_i^* (f (b_j))$ its {\em matrix coefficients} (where $(c_i^*)_{i\in I}$ 
is the dual basis of $c$).
We write also $(\phi \otimes v)(x)=\phi(x) \cdot v$.
Then
$$
f = \sum_{(i,j) \in I \times J} f_{i,j}  \,  b_j^* \otimes c_i , \qquad
f(b_k)= \sum_i f_{i,k} c_k .
$$
When
writing a matrix in the usual way as rectangular number array,
we use the natural {\em total order} on the index set -- that is, the
{\em lexicographic order}; 
for instance,
$$
\cP(\sett{1,2}) = ( \emptyset, \sett{1},\sett{2},\sett{1,2}).
$$
In the following, for an $n$-tuple 
 ${\bf a} = (a_i)_{i \in N} \in \K^n$,  we use the
notation ${\bf a}_N:= \prod_{i\in N} a_i$, in the same way as
we do for $\ttt , \sss \in \K^n$ in the main text.
When $N$ is considered to be fixed, and $A \subset N$, we denote
by $A^c = N \setminus A$ its complement.

The following result allows to put hands on induction procedures 
using iterated tensor products, cf.  Remark \ref{rk:P-tensorproduct}.

\begin{theorem}\label{th:hyper-matrix}
Let $N = \{ k_1 , \ldots , k_n \}$ and $f_i : \K^{\cP(\sett{k_i})} \to 
\K^{\cP(\sett{k_i})}$ linear, with matrix
$$
f_i = \begin{pmatrix} a_i & b_i  \\ c_i & d_i \end{pmatrix}:
\qquad
E_\emptyset^i  \mapsto a_i E_\emptyset^i + c_i E_i^i,  \quad
E_i^i \mapsto b_i E_\emptyset^i + d_i E_i^i.
$$
Then the matrix of the linear map
$f:=\otimes_{i=1}^n f_i : \K^{\cP(N)} \to \K^{\cP(N)}$
is given by the matrix coefficients, for $(A,B) \in \cP(N)^2$, 
$$
f_{A,B} = E_A^* \bigl(f (E_B) \bigr) = 
{\bf a}_{A^c \cap B^c} \cdot {\bf b}_{A^c \cap B} \cdot
{\bf c}_{A \cap B^c} \cdot {\bf d}_{A \cap B} .
$$
In other terms,
$f(E_B^N) =\sum_{A \in \cP(N)} {\bf a}_{A^c \cap B^c} \cdot {\bf b}_{A^c \cap B} \cdot
{\bf c}_{A \cap B^c} \cdot {\bf d}_{A \cap B} \, \,  E_A^N$,
or
$$
f = \sum_{(A,B)\in \cP(N)^2} {\bf a}_{A^c \cap B^c} \cdot {\bf b}_{A^c \cap B} \cdot
{\bf c}_{A \cap B^c} \cdot {\bf d}_{A \cap B} \, \,
(E_B^N)^* \otimes E_A^N .
$$
\end{theorem}

\begin{proof}
When the cardinality $n$ of $N$ is equal to one, then the claim
is true, directly by definition of the matrix coefficients.
For $n=2$, the matrix of $f_1 \otimes f_2$ is
$$
\begin{pmatrix} a_1 & b_1 \\ c_1 & d_1 \end{pmatrix} \otimes 
\begin{pmatrix} a_2 & b_2 \\ c_2 & d_2 \end{pmatrix} =
\begin{pmatrix}
a_1 a_2 & b_1 a_2 & a_1 b_2 & b_1 b_2 \\
c_1 a_2 & d_1 a_2 & c_1 b_2  & d_1 b_2 \\
a_1 c_2 & b_1 c_2 & a_1 d_2 & b_1 d_2 \\
c_1 c_2 & d_1 c_2 & c_1 d_2 & d_1 d_2  \end{pmatrix}
$$
(``Kronecker product''). 
For instance, when $B = \emptyset$, so $B^c = \sett{1,2}$,
$$
f(E_\emptyset^\sett{12}) = a_{12} E_\emptyset + c_1 a_2 E_1 +
a_1 c_2 E_2 + c_{12} E_{12},
$$
in keeping with the claim.
In the general case, we expand the expression
$$
f = \otimes_i f_i = \otimes_i
\Bigl( 
a_i (E_\emptyset^i)^* \otimes E_\emptyset^i + 
b_i (E_\emptyset^i)^* \otimes E_i^i +
c_i (E_i^i)^* \otimes E_\emptyset^i  +
d_i (E_i^i)^* \otimes E_i^i
\Bigr) 
$$
by distributivity: we get a sum of $4^n$ terms, which correspond 
exactly to the $4^n$ terms in the last formula of the claim. 
(E.g., for $n=2$, there are $16$ terms, corresponding to expanding the
product
$(a_1+b_1 + c_1 + d_1)(a_2 + b_2 + c_2 +d_2)$
by distributivity, giving the $16$ matrix coefficients shown above. 
The first column contains the $4$ terms from expanding
$(a_1 + c_1)(a_2+c_2)$, etc.)
\end{proof}

To memorise the formula: for $2\times 2$-matrices and indices,
the correspondence is
$$
\begin{pmatrix} {\bf a} & {\bf b} \\ {\bf c }&{\bf  d} \end{pmatrix} \qquad : \qquad
\begin{pmatrix} A^c \cap B^c & A^c \cap B \\ A \cap B^c & A \cap B 
\end{pmatrix} .
$$
Next, we give a formula for the {\em inverse} of $f$, when its determinant
is invertible.
 \href{https://en.wikipedia.org/wiki/Kronecker_product#Relations_to_other_matrix_operations}{From
well-known properties of the Kronecker product} it follows that
$$
\det(f) = \det(\otimes_{i=1}^n f_i) = (\prod_{i=1}^n \det(f_i))^{2^{n-1}},
$$
whence the first statement of the following theorem:

\begin{theorem}\label{th:hyper-inverse}
Let $N$ and $f = \otimes_{i=1}^n f_i$ be as in the preceding theorem.
Then $f$ is invertible if, and only if, all $f_i$ are invertible, and then
its inverse is given by the matrix coeffients, for 
$(A,B)\in \cP(N)^2$ (recall $A \Delta B$ is the symmetric difference) 
$$
(f^{-1})_{A,B} = \frac{(-1)^{\vert A \Delta B\vert}}{\prod_{i=1}^n \det(f_i) } 
f_{B^c,A^c} =
\frac{(-1)^{\vert A \Delta B\vert}}{\prod_{i=1}^n \det(f_i) } 
{\bf a}_{A \cap B} \cdot {\bf b}_{A \cap B^c} \cdot
{\bf c}_{A^c \cap B} \cdot {\bf d}_{A^c \cap B^c} .
$$
\end{theorem}

\begin{proof} Assume each $f_i$ is invertible. 
For $n=1$, $N = \sett{k}$, the inverse is
\begin{equation}\label{eqn:inverse2}
\begin{pmatrix} a_k & b_k \\ c_k & d_k\end{pmatrix}^{-1} =
\frac{1}{(a_k d_k - b_k c_k)}
 \begin{pmatrix} 
d_k & - b_k \\ -c_k & a_k
\end{pmatrix} .
\end{equation}
For $n=2$, the matrix of the inverse is the
Kronecker product of the inverses
$$
\frac{1}{\det(f_1)\det(f_2)}
 \begin{pmatrix} 
d_1 & - b_1 \\ -c_1 & a_1
\end{pmatrix} \otimes 
 \begin{pmatrix} 
d_2 & - b_2 \\ -c_2 & a_2
\end{pmatrix} =
$$
$$
\frac{1}{\det(f_1)\det(f_2)}
 \begin{pmatrix} 
 d_1 d_2 & -b_1 d_2 &- d_1 b_2 & b_1 b_2 \\
-c_1 d_2 & a_1 d_2 & c_1 b_2 & - a_1 b_2 \\
-d_1 c_2 & b_1 c_2 & d_1 a_2 & -b_1 a_2 \\
c_1 c_2  & -a_1 c_2 & - c_1 a_2  & a_1 a_2 
\end{pmatrix}  
$$
which is in keeping with the formula announced in the claim.
To put this computation into a conceptual framework, note that the inverse
in (\ref{eqn:inverse2}) is obtained by first taking the adjugate matrix, and then dividing by the determinant. 
The adjugate $X^\sharp$ of a $2 \times 2$-matrix $X$, in turn, is given by
$$
X^\sharp = J X^\top J^{-1},
$$
where $X^\top$ is the transposed matrix, $(X^\top)_{(A,B)} = X_{(B,A)}$, 
 and 
\begin{equation}\label{eqn:IJK}
I := \begin{pmatrix} 1 & 0 \\ 0 & -1 \end{pmatrix},\quad
J := \begin{pmatrix} 0 & 1 \\ - 1 & 0 \end{pmatrix}, \quad
K := \begin{pmatrix} 0 & 1 \\ 1 & 0 \end{pmatrix}, \quad
\end{equation}
i.e., $J$ sends  
 $E_\emptyset \mapsto E_1 , \, E_1 \mapsto - E_\emptyset$
(so $X^\sharp$ is the adjoint of $X$ with respect to the canonical 
symplectic form on $\K^2$; call it ``symplectic adjoint'').
For each $2 \times 2$-matrix $M$ let
$$
M_\sfn = \otimes_{i=1}^n M : \K^{\cP(\sfn)} \to \K^{\cP(\sfn)}.
$$
Then, for the matrices $I,J,K$ defined by (\ref{eqn:IJK}), the
effect on $E_A$ is
\begin{equation}
I_\sfn (E_A) = (-1)^{\vert A\vert} E_A, \quad
K_\sfn(E_A)= E_{A^c}, \quad
J_\sfn (E_A) =  (-1)^{\vert A^c \vert} E_{A^c} , 
\end{equation}
The inverse of $J_\sfn$ is $J_\sfn^{-1}(E_A) = K_\sfn I_\sfn (E_A)
= (-1)^{\vert A \vert} E_{A^c} =
(-1)^n J_\sfn (E_A)$.
Using this, we compute
\begin{align*}
f^\sharp (E_A) & =
J_\sfn \circ f^\top \circ   J_\sfn^{-1} (E_A) = 
(-1)^{\vert A\vert}  J_\sfn \circ f^\top  (E_{A^c} )
\\
& = (-1)^{\vert A\vert} J_n \sum_B f^\top_{A^c,B} E_B 
\\
& 
= (-1)^{\vert A \vert} \sum_B f_{B,A^c} (-1)^{\vert B^c \vert} E_{B^c}
= (-1)^{\vert A \vert} \sum_B f_{B^c,A^c} (-1)^{\vert B \vert} E_{B}
\\ 
&=
\sum_B (-1)^{\vert A \vert} (-1)^{\vert B \vert} 
{\bf a}_{A \cap B} \cdot {\bf b}_{A \cap B^c} \cdot
{\bf c}_{A^c \cap B} \cdot {\bf d}_{A^c \cap B^c}
 E_B
\end{align*}
which together with $\vert A \vert + \vert B \vert \equiv \vert A \Delta B \vert
\mod(2)$, so
$(-1)^{\vert A \vert} (-1)^{\vert B\vert}  =
(-1)^{ \vert A \Delta B \vert}$, gives us the adjugate and the claim.
\end{proof} 

\begin{remark}
In the same way, it follows that, even if $f$ is not invertible, we have
$$
f \circ J_\sfn \circ  f^\top \circ J^{-1}_\sfn =
\prod_{i=1}^n \det(f_i) \cdot  \id .
$$
\end{remark}

\section{On the structure of tangent algebras}\label{app:algebras}

One may be interested in defining a class of algebras, 
generalizing the by now classical {\em Weil algebras} (see \cite{KMS93, MR}),
and the {\em bundle algebras} from \cite{Be14},  incorporating
also algebras arising from difference calculus. 
The following structure theorem might help to select  structural 
features that could be used for defining such a category. 
We use notation defined in Subsection \ref{ssec:alphabeta}.

\begin{theorem}[Structure of the first order tangent
algebra $ \K_{(t,s)}^\sett{1}$]\label{th:structure}$ $
\begin{enumerate}
\item
The ideals $\ker(\alpha)$ and $\ker(\beta)$ satisfy $\ker(\alpha) \cdot \ker(\beta)=0$.
\item
The product of  $w,v \in \K_{(t,s)}^\sett{1}$  is given by
the ``fundamental relation''
$$
w \cdot v = \alpha(w)v - \alpha(w) \beta(v) + \beta(v) w.
$$
\item
The map 
$$
\ka  : \K_{(t,s)}^\sett{1}\to  \K_{(t,s)}^\sett{1} , \quad
v \mapsto (\alpha + \beta)(v) \cdot 1 - v 
$$
is an algebra automorphism of order $2$ such that $\alpha \circ \ka = \beta$.
Moreover,
$$
\forall v \in  \K_{(t,s)}^\sett{1} : \quad
v \cdot \ka(v) = \alpha(v) \beta(v) 1 .
$$
\item
An element $v$ is invertible in $ \K_{(t,s)}^\sett{1}$ if, and only if,
$\alpha(v)\beta(v) \in \K^\times$, and then the inverse is
$$
v^{-1} = \frac{1}{\alpha(v)\beta(v)} \kappa(v) =
(\frac{1}{\alpha(v)} + \frac{1}{\beta(v)}) 1 - \frac{v}{\alpha(v) \beta(v)}.
$$
\item
The set $\K_{(t,s)}^\sett{1}$, equipped with the following product $\ast$ (for
$(u,w)$ such $\alpha(u)=\beta(w)$), inversion $\kappa$, 
and units $\lambda 1$ ($\lambda \in \K$),
is a groupoid:
$$
u \ast w = u - \alpha(u) 1 + w.
$$
\end{enumerate}
\end{theorem}

\begin{proof}
(1)
$\ker(\alpha) = \K (e - s)$ and $\ker(\beta)= \K (e-t)$, and, by the defining relation of the algebra,
$(e-s)(e-t) = [(X-t)(X-s)]=0$ .

(2)
Since $\alpha(v-\alpha(v)1)= 0$ and $\beta(w-\beta(w)1)=0$, the preceding item implies
$$
0 = (v - \alpha(v)) (w - \beta(w)) = vw - \alpha(v)w - \beta(w)v + \alpha(v) \beta(w). 
$$

(3)
Note that $\kappa(1)=1+1-1=1$ and $\kappa(e) = s+t - e$, whence
$\kappa(\kappa(e))=s+t - (s+t - e)=e$, so $\kappa^2 = \id$. 
Next,
$$
\alpha (\kappa(v) )= (\alpha + \beta)(v) - \alpha(v) = \beta(v) .
$$
To prove that $\kappa$ is an automorphism, since $\kappa(1)=1$, it suffices to show that
$\kappa(e^2) = \kappa(e)^2$. 
Indeed,
$\kappa(e)^2 = (t+s)^2 - 2 (t+s)e + e^2 = (t+s)^2 - ts - (t+s)e$
and $\kappa(e^2) = \kappa(-ts + (t+s)e) = - ts + (t+s) \kappa(e) =
-ts + (t+s)^2 - (t+s) e$.
Finally,
$$
v \cdot \kappa(v) = 
\alpha(v) \kappa(v) - \alpha(v) \beta(\kappa v) + \beta(\kappa v) v =
\alpha(v) \beta(v)1 .
$$

(4) If $v$ is invertible, then applying the morphisms $\alpha$ and $\beta$,
it follows that both $\alpha(v)$ and $\beta(v)$ are invertible.
Conversely, the last formula from (3) shows that under this condition
$v$ has an inverse given by $v^{-1}$ as in the claim.

(5) The defining properties of a groupoid are easily checked by direct
computation, cf.\ \cite{Be15, Be17}.
\end{proof}


It is then true, moreover, that the groupoid law $\ast$ is an algebra
morphism from the fiber product algebra
$\K_{(t,s)}\times_{\alpha,\beta} \K_{(t,s)}$ to $\K_{(t,s)}$, and thus is
``internal'' to a certain category of algebras.

\end{document}